\documentclass{amsart}

\usepackage{amsmath,amsfonts,amsthm,amssymb,mathrsfs,latexsym,paralist,xypic}
\usepackage{enumerate}
\usepackage{nomencl,mathrsfs}
\usepackage{color,pdfpages}
\usepackage{tikz, soul}
\usepackage{lipsum}
\usepackage{dsfont}
\usepackage{mathabx}
\usepackage[margin=3cm]{geometry}
\usepackage{hyperref}
\usepackage[OMLmathsfit]{isomath}
\usepackage{subcaption}

\usetikzlibrary{shapes.arrows}

\newtheorem{theorem}{Theorem}[section]

\newtheorem{proposition}[theorem]{Proposition}

\newtheorem{corollary}[theorem]{Corollary}

\theoremstyle{definition}
\newtheorem{definition}[theorem]{Definition}
\newtheorem{remark}[theorem]{Remark}
\newtheorem{example}[theorem]{Example}

\numberwithin{equation}{section}

\newcommand{\co}{\mskip0.5mu\colon\thinspace}

\newcommand{\C}{\mathbb{C}}

\newcommand{\N}{\mathbb{N}}

\newcommand{\din}{\operatorname{d}^{\operatorname{in}}}
\newcommand{\dout}{\operatorname{d}^{\operatorname{out}}}
\newcommand{\bdin}{\operatorname{d}^{\operatorname{b,in}}}
\newcommand{\bdout}{\operatorname{d}^{\operatorname{b,out}}}
\newcommand{\im}{\mathbf{M}}
\newcommand{\lm}{\boldsymbol{\Delta}}
\newcommand{\am}{\mathbf{A}}
\newcommand{\bam}{\mathbf{A_{\operatorname{b}}}}
\newcommand{\eam}{\mathbf{L}}
\newcommand{\ham}{\mathbf{H}}
\newcommand{\sam}{\mathbf{S}}
\newcommand{\bsam}{\mathbf{S_{\operatorname{b}}}}
\newcommand{\slm}{\boldsymbol{\Delta}_S}
\newcommand{\bslm}{\boldsymbol{\Delta}_{\operatorname{b},S}}
\newcommand{\tpm}{\mathbf{P}}
\newcommand{\nlm}{\boldsymbol{\Delta}_N}
\newcommand{\clm}{\boldsymbol{\Delta}_C}
\newcommand{\btpm}{\mathbf{P_{\operatorname{b}}}}
\newcommand{\bnlm}{\boldsymbol{\Delta}_{\operatorname{b},N}}
\newcommand{\bclm}{\boldsymbol{\Delta}_{\operatorname{b},C}}

\newcommand{\lSpec}{\operatorname{Spec}_{\boldsymbol{\Delta}}}
\newcommand{\aSpec}{\operatorname{Spec}_{\mathbf{A}}}
\newcommand{\baSpec}{\operatorname{Spec}_{\mathbf{A_{\operatorname{b}}}}}
\newcommand{\aSpecS}{\operatorname{Spec}_{\mathbf{A}}^{\mathcal{S}} }
\newcommand{\baSpecS}{\operatorname{Spec}_{\mathbf{A_{\operatorname{b}}}}^{\mathcal{S}} }
\newcommand{\eSpec}{\operatorname{Spec}_{\mathbf{L}}}

\newcommand{\hSpec}{\operatorname{Spec}_{\mathbf{H}}}
\newcommand{\sSpec}{\operatorname{Spec}_{\mathbf{S}}}
\newcommand{\bsSpec}{\operatorname{Spec}_{\mathbf{S}_{\operatorname{b}}}}
\newcommand{\slSpec}{\operatorname{Spec}_{\boldsymbol{\Delta}_S}}
\newcommand{\bslSpec}{\operatorname{Spec}_{\boldsymbol{\Delta}_{\operatorname{b},S}}}
\newcommand{\nlSpec}{\operatorname{Spec}_{\boldsymbol{\Delta}_N}}
\newcommand{\clSpec}{\operatorname{Spec}_{\boldsymbol{\Delta}_C}}
\newcommand{\bnlSpec}{\operatorname{Spec}_{\boldsymbol{\Delta}_{\operatorname{b},N}}}
\newcommand{\bclSpec}{\operatorname{Spec}_{\boldsymbol{\Delta}_{\operatorname{b},C}}}
\newcommand{\ed}{\mathcal{L}}
\newcommand{\unp}{\mathcal{UP}}
\newcommand{\NN}{\operatorname{N}}
\newcommand{\Tr}{\operatorname{Trace}}
\newcommand{\Trn}{\text{T}}
\newcommand{\diag}{\operatorname{diag}}

\begin{document}

\title{The spectra of digraphs with Morita equivalent $C^\ast$-algebras}

\author{Carla Farsi}
\address{Department of Mathematics, University of Colorado at Boulder,
    Campus Box 395, Boulder, CO 80309-0395}
\email{farsi@euclid.colorado.edu}

\author{Emily Proctor}
\address{Department of Mathematics, Middlebury College, Middlebury, VT 05753}
\email{eproctor@middlebury.edu}

\author{Christopher Seaton}
\address{Department of Mathematics and Computer Science,
    Rhodes College, 2000 N. Parkway, Memphis, TN 38112}
\email{seatonc@rhodes.edu}



\begin{abstract}
Eilers et al. have recently completed the geometric classification of unital graph $C^\ast$-algebras up to Morita equivalence using a set of moves on the corresponding digraphs. We explore the question of whether these moves preserve the nonzero elements of the spectrum of a finite digraph, which in this paper is allowed to have loops and parallel edges. We consider several different digraph spectra that have been studied in the literature, answering this question for the Laplace and adjacency spectra, their skew counterparts, the symmetric adjacency spectrum, the adjacency spectrum of the line digraph, the Hermitian adjacency spectrum, and the normalized Laplacian, considering in most cases two ways that these spectra can be defined in the presence of parallel edges. We show that the adjacency spectra of the digraph and line digraph are preserved by a subset of the moves, and the skew adjacency and Laplace spectra are preserved by the Cuntz splice. We give counterexamples to show that the other spectra are not preserved by the remaining moves. The same results hold if one restricts to the class of strongly connected digraphs.
\end{abstract}

\maketitle

\tableofcontents

\section{Introduction}

Given a graph, undirected or directed, there are a number of ways to associate a matrix to the graph and, as a consequence, various eigenvalue spectra that can be considered.  These spectra are graph isomorphism invariants and have been studied in their own right.  For example, the spectrum of the adjacency matrix of a digraph characterizes the number of cycles in the graph; see \cite{BowenLanford}, \cite{Storm}.  See \cite{Chung}, \cite{Brualdi} for basic introductions to the subject.

Beyond the study of intrinsic properties of graphs, spectra of graphs arise in the study of other mathematical structures. For instance, the spectrum of the Laplacian on a domain or manifold is a classical geometric invariant that has been extensively studied for decades. The notion of discretizing the Laplace spectrum of a manifold has led to the investigation of approximating graphs whose spectra tends to the spectrum of the manifold Laplacian. For example, by embedding a finite graph with a given Laplace spectrum into a manifold and extending the metric globally, Colin de Verdi\`ere showed that for any closed manifold $M$, one can choose a metric on $M$ in such a way as to prescribe the first $N$ eigenvalues of the Laplacian acting on functions on $M$; see \cite{ColindeVerdiere,Jammes,EgidiPost}, and \cite{FarsiProctorSeaton} for the case of orbifolds.

As another important instance of associating an analytical structure to a digraph $D$, we can define a $C^*$-algebra $C^*(D)$ based on the information encoded in $D$.  For example, starting from either a one-vertex digraph having $N$ loops, or from the complete digraph on $N$ vertices, we can generate a $C^*$-algebra that is canonically isomorphic to the Cuntz algebra $\mathcal{O}_N$.  More generally, given an irreducible matrix $\bam$ of $0$'s and $1$'s, we can interpret $\bam$ as the adjacency matrix of a digraph $D_{\bam}$.  The $C^*$-algebra $C^*(D_{\bam})$ associated to $D_{\bam}$ is canonically isomorphic to the Cuntz-Krieger algebra $\mathcal{O}_{\bam}$. See \cite{CuntzMarkov} and \cite{CuntzMarkovII}.

This paper focuses on the relationship between the spectra of a finite digraph $D=(V,E,r,s)$ and the Morita equivalence class of its associated $C^*$-algebra.
Note that we allow $D$ to have loops and parallel edges; see Definition~\ref{def:pseudodigraph}.  The main question that we address is: given a digraph, to what extent does the Morita equivalence class of the associated $C^*$-algebra determine the spectrum of the digraph?  We note that there are numerous matrices, and thus numerous eigenvalue spectra, that one can associate to a given digraph.  We examine a wide collection here; see Definition~\ref{def:spectra}.

Our work is based on recent results on a classification of digraph $C^*$-algebras via $K$-theoretic and combinatorial invariants.  In particular, in a series of papers \cite{Res06,Sorensen,ERR10,ERS12,ERRS18} culminating in \cite{EilersRestorffEA}, Eilers, et al.~obtained a geometric classification of unital digraph $C^*$-algebras. Observe that a digraph $C^*$-algebra is unital if and only if its corresponding digraph has finitely many vertices (see
\cite[Section 2.3]{EilersRestorffEA}). The authors of these papers defined six ``moves", labeled (S), (R), (O), (I), (C), and (P), that can be performed on digraphs having countably many vertices and edges. Although these moves can be defined for countable digraphs, for their classification results, Eilers, et al.~restrict their attention to digraphs having finitely many vertices (see \cite[Theorem 3.1 and Corollaries 3.2 and 3.3]{EilersRestorffEA}). In \cite{EilersRestorffEA}, they show that the digraph $C^*$-algebras $C^*(D_1)$ and $C^*(D_2)$ associated to two digraphs $D_1$ and $D_2$ are stably isomorphic if and only if $D_1$ and $D_2$ differ by a finite sequence of these six moves and their inverses.  In this context, because $C^*(D_1)$ and $C^*(D_2)$ are separable unital $C^*$-algebras, they are stably isomorphic if and only if they are Morita equivalent.  We note further that by \cite{BatesPaskRaeburnSzymanski, FujiiWatatani}, for a digraph $D$, the digraph $C^*$-algebra $C^*(D)$ is isomorphic to $C^*(\ed(D))$, where $\ed(D)$ denotes the line digraph, also known as the dual or adjoint digraph, associated to $D$ (see Definition~\ref{def:LineSimplifiedGraph}).

A natural  question arises from the above results: How does the spectrum of a digraph behave under the above moves? For instance, the fact that the $C^\ast$-algebras of $D$ and $\ed(D)$ are isomorphic, together with the fact that the nonzero elements of the adjacency spectra of $D$ and $\ed(D)$ coincide (see \cite[Theorem~1.2]{Brualdi}, \cite[Theorem~1.4.4]{LiuLai}, and Proposition~\ref{prop:edgedualspectrum} below), indicate a potential connection. In this paper, we restrict our attention to finite digraphs, i.e. digraphs having finitely many vertices and edges, so that the spectrum is defined. For digraphs having finitely many vertices, this amounts to the requirement that the digraphs have no infinite emitters. We also consider this question with the additional restriction that the digraphs are strongly connected.

Our main result is the following, corresponding to Propositions~\ref{prop:movesadjacencyspectrum},
\ref{prop:moveslineadjacencyspectrum}, \ref{prop:movesbinaryadjacencyspectrum}, and \ref{prop:movesskewspectra},
as well as Tables~\ref{tab:ProofOrgGeneral} and \ref{tab:ProofOrgStrongConn} in Section~\ref{sec:NegResults}.

\begin{theorem}
\label{thrm:MainResult}
Let $D$ be a finite digraph, and consider the Moves (S), (R), (O), (I), (C), and (P) that preserve the Morita equivalence class of the $C^\ast$-algebra of $D$.
\begin{itemize}
\item[(i)]      Moves (S), (O), and (I) preserve the multiset of nonzero elements of the adjacency spectrum
                $\aSpec(D)$ of $D$ while Moves (R), (C), and (P) do not;
\item[(ii)]     Moves (S), (O), and (I) preserve the multiset of nonzero elements of the line adjacency spectrum
                $\eSpec(D)$ while Moves (R), (C), and (P) do not;
\item[(iii)]    Move (S) preserves the multiset of nonzero elements of the binary adjacency spectrum
                $\baSpec(D)$ while Moves (R), (O), (I), (C), and (P) do not; and
\item[(iv)]     Move (C) preserves the multisets of nonzero elements of the skew adjacency spectrum
                $\sSpec(D)$, the binary skew adjacency spectrum $\bsSpec(D)$, the
                skew Laplace spectrum $\slSpec(D)$, and the binary skew Laplace spectrum $\bslSpec(D)$,
                while Moves (S), (R), (O), (I), and (P) do not.
\end{itemize}
The (nonzero elements of the) other spectra $\lSpec(D)$, $\aSpecS(D)$, $\baSpecS(D)$, $\hSpec(D)$,
are not preserved by any of the moves.

If $D$ is a strongly connected digraph, then only the Moves (R), (O), (I), and (C) can be applied to $D$, and these claims remain true. Furthermore, the spectra $\nlSpec(D)$, $\bnlSpec(D)$, $\clSpec(D)$, and $\bclSpec(D)$, which are only defined for strongly connected digraphs, are not preserved by any of these four moves.
\end{theorem}

Note that if a move preserves a spectrum, then it is clear that the inverse of the move also preserves the spectrum.

The proofs of parts (i), (ii), and (iii) of Theorem~\ref{thrm:MainResult} rely on the characterization of the adjacency spectrum and line adjacency spectrum of a digraph in terms of the number of cycles of a given length that the digraph contains.  The key idea of each proof consists of showing that the number of cycles of a given length is preserved by the appropriate moves.  Part (iv) follows in a straightforward way from the definition of Move (C).  We provide counterexamples to show that none of the other spectra are preserved.  Because of the number of spectra and moves considered, Tables~\ref{tab:ProofOrgGeneral} and \ref{tab:ProofOrgStrongConn} summarize the result or counterexample that proves each case of Theorem~\ref{thrm:MainResult}.
Observe that parts (i) and (ii) of Theorem~\ref{thrm:MainResult}, are consistent with the observation above that for a given digraph $D$, $C^*(D)$ coincides with $C^*(\ed(D))$.


Finally, we observe that a question that is beyond the scope of this paper and that remains open for further investigation is: does there exist as sequence of digraph moves that characterizes digraph isospectrality? For instance, some such moves have been identified for the Hermitian spectrum in \cite{GuoMohar,MoharHermSwitching}.

This paper is organized as follows.  In Section~\ref{sec:Spectra}, we set the digraph terminology that we will use throughout the paper and catalog the various matrices and spectra whose relation to Morita equivalence we will explore.  Sections~\ref{sec:Cycles} and \ref{sec:MovesMorita} contain the positive results of our paper.  The results of Section~\ref{sec:Cycles} allow us to reduce the questions of adjacency isospectrality and line adjacency isospectrality to an enumeration of cycles of a given length. In Section~\ref{sec:MovesMorita}, we use these results to show which digraph moves do preserve these spectra, proving Propositions~\ref{prop:movesadjacencyspectrum},
\ref{prop:moveslineadjacencyspectrum}, \ref{prop:movesbinaryadjacencyspectrum}, and \ref{prop:movesskewspectra}.  Section~\ref{sec:NegResults} contains counterexamples that support our negative results.  In particular, these examples show which of the various spectra are \emph{not} preserved by the various digraph moves.


Many of the examples in Section~\ref{sec:NegResults} were produced using \emph{Mathematica} \cite{Mathematica}.   \emph{Mathematica} notebooks implementing the moves and spectra considered here are available from the authors upon request.

\subsection{Acknowledgements}
							
C.F. would like to thank the sabbatical program at the University of Colorado at Boulder and was partially supported by the Simons Foundation Collaboration Grant for Mathematicians \#523991.
C.S. would like to thank the sabbatical program at Rhodes College and was partially supported by the E.C. Ellett Professorship in Mathematics.
E.P. would like to thank the sabbatical program at Middlebury College.


\section{Digraphs and Their Spectra}
\label{sec:Spectra}

\subsection{Background and notation for digraphs}
\label{subsec:DigraphBackground}

We focus in this paper on \emph{finite digraphs}, i.e., directed graphs with finitely many vertices
and finitely many edges that may have loops and parallel edges. We clarify the language with the following.

\begin{definition}[digraph, simple digraph, multidigraph]
\label{def:pseudodigraph}
A \emph{digraph} $D = (V,E,r,s)$ consists of a set of \emph{vertices} $V$, a set of
\emph{edges} $E$, and functions $s\co E\to V$ and $r\co E\to V$ called the \emph{source} and \emph{range},
respectively. In a fixed digraph $D$, a \emph{loop} is an edge $e\in E$ such that $r(e) = s(e)$, and
two edges $e \neq f$ are \emph{parallel} if $s(e) = s(f)$ and $r(e) = r(f)$. A \emph{simple directed
graph} or \emph{simple digraph} is a digraph with no loops nor parallel edges, and a \emph{multidigraph}
is a digraph that may have parallel edges but contains no loops. A digraph is \emph{finite} if $V$
and $E$ are both finite sets.
\end{definition}

Note that, for simplicity, we use the term \emph{digraph} to refer to what is sometimes called a \emph{pseudodigraph}, as some authors require that digraphs have no loops nor multiple edges. We will use \emph{simple digraph} when we would like to emphasize the absence of loops and multiple edges, and \emph{multidigraph} when we would like to emphasize the absence of loops. For simplicity, we will sometimes say that an edge $e\in E$ such that $s(e) = v$ and $r(e) = w$ is an edge \emph{from $v$ to $w$}. We use the notation $\lvert S \rvert$ to denote the cardinality of the set $S$.

Following \cite{EilersRestorffEA17}, we have the following definitions for a digraph $D=(V,E,r,s)$.

\begin{definition}[path, cycle, simple cycle, vertex-simple cycle, exit, return path]
\label{def:paths}
A \emph{path of length $n$} is a finite sequence $(e_1,\dots,e_n)$ of edges with $r(e_i)=s(e_{i+1})$ for $i=1,\dots n-1$.
Note that it is possible that $e_j=e_k$ for $j\neq k$.
A \emph{cycle} is a nonempty path $(e_1,\dots,e_n)$ such that $r(e_n)=s(e_1)$.  A cycle of length $n$ is \emph{simple} if $e_i\neq e_j$ for any $i,j=1,\dots,n$ with $i\neq j$ and it is \emph{vertex-simple} if $r(e_i)\neq r(e_j)$ for any $i,j=1,\dots,n$ with $i\neq j$.  We say that a vertex-simple cycle $(e_1,\dots,e_n)$ has an \emph{exit} if there is an edge $f$ such that $s(f)=s(e_i)$ for some $i=1,\dots, n$ with $f\neq e_i$.  A cycle $(e_1,\dots,e_n)$ is a \emph{return path} if $r(e_n)\neq r(e_i)$ for any $i=1,\dots, n-1$.
\end{definition}

We also have the following for vertices.

\begin{definition}[source, sink, regular vertex]
\label{def:vertices}
A \emph{source} is a vertex $v$ such that $r^{-1}(v)=\emptyset$.  A \emph{sink} is a vertex $v$ such that $s^{-1}(v)=\emptyset$.  A \emph{regular vertex} is a vertex for which $s^{-1}(v)$ is finite and nonempty.
\end{definition}

We will make use of the following two constructions of digraphs from a given digraph.
Note that the line digraph is often called the \emph{dual} or \emph{adjoint digraph} in the literature on
$C^\ast$-algebras of digraphs, \cite{BatesPaskRaeburnSzymanski}, \cite[p.~237]{MannRaeburnEA}, \cite{FujiiWatatani}, while the term
\emph{line digraph} appears in the graph theory literature \cite[p.~2182]{Brualdi},
\cite[p.~173]{Jovanovic}.

\begin{definition}[line digraph, unparalleled digraph]
\label{def:LineSimplifiedGraph}
Let $D = (V,E,r,s)$ be a digraph.
\begin{enumerate}
\item[(i)]  The \emph{line digraph} $\ed(D)$ associated to $D$ is the digraph with vertex set $E$ and
            edge set given by the set of composable pairs of edges. That is, the edge set of $\ed(D)$
            is the set of pairs $(e,f)\in E^2$ such that $r(e) = s(f)$, with
            $s_{\ed(D)}\big((e,f)\big) = e$ and $r_{\ed(D)}\big((e,f)\big) = f$.
\item[(ii)] Define an equivalence relation $\sim$ on $E$ by saying $e\sim f$ if $s(e) = s(f)$ and $r(e) = r(f)$.
            The \emph{unparalleled digraph} $\unp(D)$ associated to $D$ is the digraph with vertex set
            $V$ and edge set $E/\sim$, where the source and range of $\unp(D)$ are those inherited from $D$.
\end{enumerate}
\end{definition}

It follows from the definition that $\ed(D)$ equals $\unp(\ed(D))$ for any digraph $D$, as $\ed(D)$ may have loops but has no parallel edges.

We will sometimes restrict our attention to the following class of digraphs.

\begin{definition}[strongly connected digraph]
\label{def:StronglyConn}
A digraph $D = (V,E,r,s)$ is \emph{strongly connected} if for each pair of vertices $v,w\in V$, there
is a path from $v$ to $w$.
\end{definition}

Finally, we will use the following various notions of degree for elements of $V$ based on those defined in
\cite{Brualdi}, \cite[p.~53]{BrualdiRyser}, and \cite{ChungDiameter}. Note that our notions of binary indegree and binary outdegree correspond to the indegree and outdegree defined in \cite{ChungCheeger}.

\begin{definition}[indegree, outdegree, binary indegree, binary outdegree]
\label{def:InOutDegree}
Let $D = (V,E,r,s)$ be a finite digraph. If $v \in V$, the \emph{indegree of $v$}, denoted $\din_v$,
is $\lvert r^{-1}(v)\rvert$, the number of edges with range
$v$, and the \emph{outdegree of $v$} $\dout_v$, is $\lvert s^{-1}(v)\rvert$, the number of edges with source $v$.
The \emph{binary indegree of $v$} denoted $\bdin_v$, is $\big\lvert s\big(r^{-1}(v)\big)\big\rvert$, the number
of distinct $w$ such that there is an edge from $w$ to $v$, and the \emph{binary outdegree of $v$} $\bdout_v$, is
$\big\lvert r\big(s^{-1}(v)\big)\big\rvert$, the number of distinct vertices $w$ such that there is an edge from $v$ to $w$.
\end{definition}

If $D$ is simple, it is easy to see that $\bdin_v = \din_v$ and $\bdout_v = \dout_v$ for each $v\in V$.
For a general digraph $D$, the binary indegree $\bdin_v$ of $v\in V$ is the indegree of the vertex $v$ in
the unparalleled graph $\unp(D)$, and similarly $\bdout_v$ in $D$ is equal to $\dout_v$ with respect to $\unp(D)$.

\subsection{Matrices and spectra associated to digraphs}
\label{subsec:MatricesSpectra}

Suppose that $D = (V,E,r,s)$ is a finite digraph. There are a variety of matrices and corresponding eigenvalue
spectra that have been associated to $D$. Before stating the formal definitions, let us briefly discuss the
appearances of these spectra and explain our terminology. Note that the authors of some of the publications
cited below consider only simple digraphs, multidigraphs, etc., but the definitions extend
readily to the case of a general finite digraph. In some cases, we consider two such generalizations,
one binary and one non-binary, as described below.

The matrices most commonly associated to $D$ are the \emph{Laplacian} or \emph{Kirchhoff matrix}, defined in terms
of the \emph{incidence matrix} (see \cite[Def.~4.2, Prop~4.8]{BiggsBook} and \cite[Def.~9.5, Lem.~9.6]{StanleyAlgCombBook}),
and the \emph{adjacency matrix}. Authors differ on whether the adjacency matrix takes parallel edges into consideration
(\cite[p.~53, Sec.~3.6]{BrualdiRyser}, \cite[p.~1]{LiuLai}), in which case the values of the adjacency matrix are
nonnegative integers; or ignores parallel edges, (\cite{HKMR71}, \cite{KrishnamoorthyParthasarathy},
\cite[Sec.~1.7]{ButlerThesis}), in which case the entries are elements of $\{0,1\}$. For our purposes, both cases will be
of interest. Specifically, the moves we will consider in Section~\ref{sec:MovesMorita} can add or remove parallel edges;
an example is illustrated
in Figure~\ref{fig:D3Counterex}, where the application of Move (O) splits the parallel loops at vertex $v_2$ (and hence
the inverse of Move (O) can introduce parallel loops). Following the convention established in
Definition~\ref{def:InOutDegree}, we will use the term \emph{binary} to indicate a matrix, spectrum, degree, etc. that ignores parallel edges and hence depends only on the unparalleled digraph $\unp(D)$. Hence, we consider both
the \emph{adjacency matrix} and the \emph{binary adjacency matrix}. We also consider the
\emph{symmetric adjacency spectrum}, that of the product of the adjacency matrix and its transpose,
which was studied in \cite{Jovanovic} and appeared in \cite{Brualdi} as the singular value decomposition of the adjacency matrix.
This can be defined in a binary and non-binary sense as well. The adjacency matrix of the line digraph $\ed(D)$, here
called the \emph{line adjacency matrix}, has appeared for instance in \cite[p.~2182]{Brualdi}, and we will see that its
spectrum is closely related to that of the adjacency matrix. One could also consider a binary line adjacency matrix
as the adjacency matrix of the line digraph associated to the unparalleled digraph $\unp(D)$, but we consider this
unmotivated and redundant, because a consequence of \cite[Theorem~1.2]{Brualdi}, \cite[Theorem~1.4.4]{LiuLai},
or Proposition~\ref{prop:edgedualspectrum} below
is that this spectrum coincides with the spectrum of the binary adjacency matrix up to the addition of zeros.
As the line digraph has no multiple edges, its adjacency matrix is equal to its binary adjacency matrix, so the
other interpretation of a ``binary line adjacency matrix" is as well redundant.

More recently, the \emph{Hermitian adjacency matrix} and its spectrum were introduced in \cite{GuoMohar,LiuLi}
and studied further in \cite{MoharHerm,WissingDam}. The graphs considered in \cite{GuoMohar,MoharHerm}
are \emph{mixed graphs}, which have both directed and undirected edges. However, as noted in
\cite[Sec.~2.1]{WissingDam}, each undirected edge can be replaced with two directed edges, one in each direction,
yielding a digraph with the same Hermitian adjacency matrix. The \emph{skew adjacency matrix}
\cite{CaversCioabaEA,Ganie} and related \emph{skew Laplacian} \cite{Ganie} are defined for digraphs formed
by orienting the edges of a simple (unoriented) graph, and hence are binary in our terminology and also ignore
pairs of directed edges in opposite directions. However, they admit a non-binary generalization to arbitrary
finite digraphs, and we consider both cases.

A final recent addition to the literature is the \emph{normalized Laplacian} and related \emph{combinatorial Laplacian}
\cite{ChungCheeger,ChungDiameter} and \cite[Sec.~5.4.1]{ButlerThesis}. These matrices are defined for strongly
connected digraphs, and their definition relies on this hypothesis in an essential way. Chung and Butler
consider \emph{weighted} digraphs, but here we consider the two cases that are intrinsic to $D$, the case where
each edge has weight $1$, and the binary case where, from every set of parallel edges, one edge has weight $1$
and the others have weight $0$.

We collect the formal definitions of these matrices in the following.

\begin{definition}[matrices associated to a digraph]
\label{def:Matrices}
Let $D = (V,E,r,s)$ be a finite digraph, and fix linear orders of $V$ and $E$.
\begin{enumerate}
\item[(i)]
The \emph{incidence matrix} $\im(D) = (m_{ve})_{v\in V, e\in E}$ of $D$ is the matrix whose rows are indexed by $V$, columns are indexed by $E$, and whose entries are given by
\[
    m_{ve}  =   \begin{cases}
                    +1,     &       s(e) = v \neq r(e),
                    \\
                    -1,     &       r(e) = v \neq s(e),
                    \\
                    0       &       \text{otherwise}.
                \end{cases}
\]
The \emph{Laplacian} $\lm(D)$, also called the \emph{Kirchhoff matrix}, is the matrix
defined by $\lm(D) = \im(D)\im(D)^\Trn$. Some authors define $\im(D)$ to be the negative of that
given here, but clearly $\lm(D)$ is independent of this choice.

\item[(ii)]
The \emph{adjacency matrix} $\am(D) = (a_{vw})_{v,w\in V}$ of $D$ is the square matrix whose rows and columns are indexed by $V$ such that $a_{vw} = \lvert\big(s^{-1}(v)\cap r^{-1}(w)\big)\rvert$, the number of edges $e\in E$ with $s(e) = v$ and
$r(e) = w$.

\noindent
The \emph{binary adjacency matrix} $\bam(D) = (a_{\operatorname{b},vw})_{v,w\in V}$ of $D$ is given by
$\bam(D) = \am\big(\unp(D)\big)$.
That is, $\bam(D)$ is the square matrix whose rows and columns are indexed by $V$ such that $a_{vw} = +1$ if there is an edge $e\in E$ with $s(e) = v$ and $r(e) = w$, and $0$ otherwise.

\item[(iii)]
The \emph{line adjacency matrix} $\eam(D) = (b_{ef})_{e,f\in E}$ of $D$ is given by $\eam(D) = \am\big(\ed(D)\big)$.
That is, $\eam(D)$ is the square matrix whose rows and columns are indexed by $E$ such that $b_{ef} = +1$
if $r(e) = s(f)$ and $0$ otherwise.

\item[(iv)]
The \emph{Hermitian adjacency matrix} $\ham(D) = (h_{vw})_{v,w\in V}$ of $D$ is the square matrix whose rows and columns are indexed by $V$ such that
\[
    a_{vw}  =   \begin{cases}
                    +1,     &       s^{-1}(v)\cap r^{-1}(w)\neq\emptyset\text{ and }s^{-1}(w)\cap r^{-1}(v)\neq\emptyset,
                    \\
                    +i,     &       s^{-1}(v)\cap r^{-1}(w)\neq\emptyset\text{ and }s^{-1}(w)\cap r^{-1}(v)=\emptyset,
                    \\
                    -i,     &       s^{-1}(v)\cap r^{-1}(w)=\emptyset\text{ and }s^{-1}(w)\cap r^{-1}(v)\neq\emptyset,
                    \\
                    0,      &       \text{otherwise}.
                \end{cases}
\]

\item[(v)]
The \emph{skew adjacency matrix} $\sam(D) = (s_{vw})_{v,w\in V}$ of $D$ is the square matrix whose rows and columns are indexed by $V$ such that $s_{vw} = \big\lvert s^{-1}(v)\cap r^{-1}(w)\big\rvert - \big\lvert s^{-1}(w)\cap r^{-1}(v)\big\rvert$.
That is, $\sam(D) = \am(D) - \am(D)^\Trn$.

\noindent
The \emph{binary skew adjacency matrix} $\bsam(D) = (s_{\operatorname{b},vw})_{v,w\in V}$ of $D$ is given by
$\bsam(D) = \sam\big(\unp(D)\big)$.
That is, $\bsam(D)$ is the square matrix whose rows and columns are indexed by $V$ such that
\[
    s_{\operatorname{b},vw}
                =   \begin{cases}
                        +1,     &   s^{-1}(v)\cap r^{-1}(w)\neq\emptyset\text{ and }s^{-1}(w)\cap r^{-1}(v)=\emptyset,
                        \\
                        -1,     &   s^{-1}(v)\cap r^{-1}(w)=\emptyset\text{ and }s^{-1}(w)\cap r^{-1}(v)\neq\emptyset,
                        \\
                        0,      &   \text{otherwise},
                    \end{cases}
\]
and hence $\bsam(D) = \bam(D) - \bam(D)^\Trn$.

\item[(vi)]
The \emph{skew Laplacian matrix} $\slm(D)$ is given by $\diag(\dout_v - \din_v) - \sam(D)$ where
$\diag(\dout_v - \din_v)$ is the diagonal matrix with rows and columns indexed by $V$ whose $vv$-entry
is the difference between the outdegree and indegree of $v$.

\noindent
The \emph{binary skew Laplacian matrix} $\bslm(D)$ is given by $\diag(\bdout_v - \bdin_v) - \bsam(D)$ where
$\diag(\bdout_v - \bdin_v)$ is the diagonal matrix with rows and columns indexed by $V$ whose $vv$-entry
is the difference between the binary outdegree and binary indegree of $v$.

\item[(vii)]
The \emph{transition probability matrix} $\tpm(D) = (p_{vw})_{v,w\in V}$ of $D$ is the square matrix whose rows and columns are indexed by $V$ and such that
\[
        p_{vw}  =   \begin{cases}
                        \frac{\lvert s^{-1}(v)\cap r^{-1}(w)\rvert}{\dout_v},
                                                &   \text{there is an edge from $v$ to $w$},
                        \\
                        0,                      &   \text{otherwise}.
                    \end{cases}
\]
Note that $p_{vw}$ is the probability of moving from $v$ to $w$ if each edge is equally likely.

Assume $D$ is strongly connected. Note that this is equivalent to the transition probability matrix $\tpm(D)$ being irreducible, which is
equivalent to the adjacency matrix $\am(D)$ being irreducible; see \cite[Theorem~3.2.1]{BrualdiRyser}. The \emph{Perron-Frobenius vector} $\phi = \phi(D)$ is the
unique left-eigenvector of $\tpm(D)$ with positive entries that sum to $1$. Let $\boldsymbol{\Phi} = \boldsymbol{\Phi}(D)$ be the diagonal matrix with entries given by those of $\phi$. The \emph{normalized Laplacian} $\nlm(D)$ is given by
\[
    \nlm(D) =   I_{\lvert V\rvert} - \frac{\boldsymbol{\Phi}^{1/2}\tpm(D)\boldsymbol{\Phi}^{-1/2} + \boldsymbol{\Phi}^{-1/2}\tpm(D)^\Trn\boldsymbol{\Phi}^{1/2}}{2},
\]
where $I_{\lvert V \rvert}$ is the $\lvert V \rvert\times\lvert V \rvert$ identity matrix, and the \emph{combinatorial Laplacian} $\clm(D)$ is given by
\[
    \clm(D) =   \boldsymbol{\Phi} - \frac{\boldsymbol{\Phi}\tpm(D) + \tpm(D)^\Trn\boldsymbol{\Phi}}{2}.
\]

\noindent
The \emph{binary transition probability matrix} $\btpm(D) = (p_{\operatorname{b},vw})$
is the square matrix whose rows and columns are indexed by $V$ and such that
\[
        p_{\operatorname{b},vw}
            =   \begin{cases}
                        \frac{1}{\bdout_v},
                                                &   \text{there is an edge from $v$ to $w$},
                        \\
                        0,                      &   \text{otherwise}.
                    \end{cases}
\]
Hence, $p_{\operatorname{b},vw}$ is the probability of moving from $v$ to $w$ in $\unp(D)$ if every
vertex in $r\big(s^{-1}(v)\big)$ is equally likely. In other words, parallel edges do not affect the likelihood
of choosing a vertex.
The definitions of the
\emph{binary normalized Laplacian} $\bnlm(D)$ and
\emph{binary combinatorial Laplacian} $\bclm(D)$ are identical to those of the normalized Laplacian
$\nlm(D)$ and combinatorial Laplacian $\clm(D)$, respectively, but with $\tpm(D)$ replaced by $\btpm(D)$.
\end{enumerate}
We will omit $D$ from the notation when it is clear from the context, e.g., $\im = \im(D)$, $\lm = \lm(D)$, etc.
\end{definition}

The spectra we consider are defined in terms of these matrices as follows.
Note that the matrices in Definition~\ref{def:Matrices} depend on the linear orders of $V$ and $E$; however, changing these orders
results in similar matrices in each case. Hence, each of the spectra defined below do not depend on this choice.
Note further that we always define the spectrum using the algebraic multiplicity of the eigenvalues.

\begin{definition}[spectra of a digraph]
\label{def:spectra}
Let $D = (V,E,r,s)$ be a finite digraph with fixed linear orders on $V$ and $E$.
\begin{enumerate}
\item[(i)]      The \emph{Laplace spectrum} $\lSpec(D)$ of $D$ is the multiset of eigenvalues of $\lm(D)$.
\item[(ii)]     The \emph{adjacency spectrum} $\aSpec(D)$ of $D$ is the multiset of eigenvalues of $\am(D)$.

                \noindent
                The \emph{symmetric adjacency spectrum} $\aSpecS(D)$ of $D$ is the multiset of eigenvalues of
                $\am(D)\am(D)^\Trn$.

                \noindent
                The \emph{binary adjacency spectrum} $\baSpec(D)$ of $D$ is the multiset of eigenvalues of $\bam(D)$.

                \noindent
                The \emph{symmetric binary adjacency spectrum} $\baSpecS(D)$ of $D$ is the multiset of eigenvalues
                of $\bam(D)\bam(D)^\Trn$.

\item[(iii)]    The \emph{line adjacency spectrum} $\eSpec(D)$ of $D$ is the multiset of eigenvalues of $\eam(D)$.

\item[(iv)]     The \emph{Hermitian adjacency spectrum} $\hSpec(D)$ of $D$ is the multiset of eigenvalues of $\ham(D)$.

\item[(v)]      The \emph{skew adjacency spectrum} $\sSpec(D)$ of $D$ is the multiset of eigenvalues of $\sam(D)$.

                \noindent
                The \emph{binary skew adjacency spectrum} $\bsSpec(D)$ of $D$ is the multiset of eigenvalues
                of $\bsam(D)$.

\item[(vi)]     The \emph{skew Laplace spectrum} $\slSpec(D)$ is the multiset of eigenvalues of $\slm(D)$.

                \noindent
                The \emph{binary skew Laplace spectrum} $\bslSpec(D)$ is the multiset of eigenvalues of $\bslm(D)$.
\end{enumerate}
Suppose further that $D$ is strongly connected.
\begin{enumerate}
\item[(vii)]    The \emph{normalized Laplace spectrum}
                $\nlSpec(D)$ of $D$ is the multiset of eigenvalues of $\nlm(D)$.

                \noindent
                The \emph{binary normalized Laplace spectrum}
                $\bnlSpec(D)$ of $D$ is the multiset of eigenvalues of $\bnlm(D)$.

                \noindent
                The \emph{combinatorial Laplace spectrum}
                $\clSpec(D)$ of $D$ is the multiset of eigenvalues of $\clm(D)$.

                \noindent
                The \emph{binary combinatorial Laplace spectrum}
                $\bclSpec(D)$ of $D$ is the multiset of eigenvalues of $\bclm(D)$.
\end{enumerate}
\end{definition}

Note that in Definition~\ref{def:spectra}(ii), we could also consider the ``right symmetric adjacency spectrum"
defined as the spectrum of $\am(D)^{\Trn}\am(D)$. However, the nonzero eigenvalues
of $\am(D)\am(D)^{\Trn}$ and $\am(D)^{\Trn}\am(D)$ coincide, so this spectrum differs from $\aSpecS(D)$
only at the multiplicity of zero; see \cite[p.~169]{Jovanovic}. The same statement holds for the binary
counterpart.

\section{The adjacency and line adjacency spectra via counting cycles}
\label{sec:Cycles}

Let $D = (V,E,r,s)$ be a digraph and let $\NN_m(D)$ denote the total number of cycles in $D$ of length $m$. We have the following, which was proven by Bowen and Lanford \cite[Theorem~1]{BowenLanford}; see also \cite[Lemma~6]{Storm}. The proof of Bowen and Lanford applies without change to the case of digraphs. In particular, though \cite{Storm} restricts to the case of a strongly connected digraph that is not a cycle, the proof applies to an arbitrary finite digraph as
long as one allows complex eigenvalues and defines the spectrum using the algebraic multiplicity of the eigenvalues.
For completeness, we include an outline of the proof.

\begin{proposition}
\label{prop:BowenLanford}
Let $D = (V,E,r,s)$ be a finite digraph. Let $M$ be the maximum modulus of elements of $\aSpec(D)$. Then
$\sum_{m=1}^\infty \frac{t^m}{m} \NN_m(D)$ converges absolutely for $\lvert t\rvert < 1/M$, and
\begin{equation}
\label{eq:CycleGenFunc}
    \exp \sum\limits_{m=1}^\infty \frac{t^m}{m} \NN_m(D)
    =   \det\big(I - t \am(D)\big)^{-1}.
\end{equation}
\end{proposition}
\begin{proof}
For $v,w\in V$, let $\NN(m,v,w)$ denote the number of walks from $v$ to $w$ of length $m$. Then
$\NN(m,v,w) = \big(\am(D)^m\big)_{vw}$ for each $m > 0$; see \cite[Theorem~3.2.1]{Stevanovic} for the case of graphs, and the proof is identical for digraphs.
Then the total number of cycles of length $m$ starting at $v\in V$ is $\NN(m,v,v) = \big(\am(D)^m\big)_{vv}$ so that
\begin{equation}
\label{eq:SumEvalPowers}
    \NN_m(D)
        =      \sum\limits_{v\in V} \NN(m,v,v)
        =    \sum\limits_{v\in V} \big(\am(D)^m\big)_{vv}
        =    \Tr\big(\am(D)^m\big)
        =    \sum\limits_{i=1}^{\lvert V \rvert} \lambda_i^m,
\end{equation}
where the $\lambda_i$ are the elements of $\aSpec(D)$.
Now, exponentiating the Taylor series
\begin{equation}
\label{eq:LogTaylor}
    -\log x = \sum\limits_{m=1}^\infty \frac{(1 - x)^m}{m}.
\end{equation}
and substituting $x = 1 - t\lambda_i$, we obtain
\begin{align*}
    \frac{1}{\prod\limits_{i=1}^{\lvert V \rvert} (1 - t \lambda_i)}
        =       \prod\limits_{i=1}^{\lvert V \rvert} \exp \sum\limits_{m=1}^\infty \frac{\lambda_i^m}{m} t^m
        =       \exp \sum\limits_{m=1}^\infty \frac{t^m}{m} \NN_m(D).
\end{align*}
As the series in Equation~\eqref{eq:LogTaylor} has a radius of convergence of $1$, the series
$\sum_{m=1}^\infty \frac{t^m}{m} \NN_m(D)$ converges absolutely to
$\log 1/\prod_{i=1}^{\lvert V \rvert} (1 - t \lambda_i)$ whenever each $\lvert t\lambda_i\rvert < 1$, completing the proof.
\end{proof}

As a consequence of Proposition~\ref{prop:BowenLanford}, the nonzero elements of the adjacency spectrum of a finite digraph $D$ are determined by the $N_m(D)$, i.e., we have the following.

\begin{corollary}
\label{cor:CycleCountInfinite}
Let $D_1=(V_1,E_1,r_1,s_1)$ and $D_2=(V_2,E_2,r_2,s_2)$ be two finite digraphs.  Suppose that $N_m(D_1)=N_m(D_2)$ for all $m$.  Then the multiset of nonzero elements of $\aSpec(D_1)$ is equal to the multiset of nonzero elements of $\aSpec(D_2)$.
\end{corollary}
\begin{proof}
The right-hand side $\det\big(I - t \am(D)\big)^{-1}$ of Equation~\eqref{prop:BowenLanford} is a rational function in $t$ and hence meromorphic on $\C$; its poles are the nonzero elements of $\aSpec(D)$ with pole order corresponding to multiplicity. Then as the left-hand side $\exp \sum_{m=1}^\infty t^m \NN_m(D)/m$ converges on a non-empty open set, by \cite[Corollary, p.~209]{RudinP}, the right-hand side is the unique analytic continuation to its domain of the left-hand side. In particular, the right-hand side is determined by the left-hand side, implying that the set of $\NN_m(D)$ determines the nonzero elements of $\aSpec(D)$.
\end{proof}

In fact, we can strengthen Corollary~\ref{cor:CycleCountInfinite} using elementary symmetric polynomials with the following.

\begin{corollary}
\label{cor:CycleCountFinite}
Let $D_1=(V_1,E_1,r_1,s_1)$ and $D_2=(V_2,E_2,r_2,s_2)$ be two finite digraphs.  Suppose that $N_m(D_1)=N_m(D_2)$ for all
$m \leq \max(\lvert V_1\rvert,\lvert V_2\rvert)$.  Then the multiset of nonzero elements of $\aSpec(D_1)$ is equal to the multiset
of nonzero elements of $\aSpec(D_2)$.
\end{corollary}
\begin{proof}
Note that by Equation~\eqref{eq:SumEvalPowers}, for $j=1,2$, the $m$th power sum $\sum_{i=1}^{\lvert V_j \rvert} \lambda_{j,i}^m$
of the eigenvalues of $\am(D_j)$ is equal to the number $\NN_m(D_j)$ of cycles of length $m$ in $D_j$.
For $j = 1, 2$, consider $\boldsymbol{\lambda}_j = (\lambda_{j,1}, \lambda_{j,2},\ldots,\lambda_{j,\lvert V_j \rvert})$ as a point in $\C^{\lvert V_j \rvert}$.
The point $\boldsymbol{\lambda}_j$ is not well-defined, as the order of the coordinates is arbitrary.
However, let $\mathcal{S}_{\lvert V_j \rvert}$ denote the symmetric group acting on $\C^{\lvert V_j \rvert}$ by permuting coordinates.
Then the $\mathcal{S}_{\lvert V_j \rvert}$-orbit of $\boldsymbol{\lambda}_j$, i.e., the set
$\mathcal{S}_{\lvert V_j \rvert}\boldsymbol{\lambda}_j = \{ g\boldsymbol{\lambda}_j : g\in \mathcal{S}_{\lvert V_j \rvert}\}$, is well-defined.
Note that we can identify $\mathcal{S}_{\lvert V_j \rvert}\boldsymbol{\lambda}_j$ with the unordered multiset of coordinates of $\boldsymbol{\lambda}_j$,
i.e., the multiset of eigenvalues $\{\lambda_{j,1},\ldots,\lambda_{j,\lvert V_j \rvert}\} = \aSpec(D_j)$.

The first $\lvert V_j \rvert$ power sums $\sum_{i=1}^{\lvert V_j \rvert} x_i^k$ for $k = 1,\ldots,\lvert V_j \rvert$
generate the ring of symmetric polynomials in $\lvert V_j \rvert$ variables
$\mathbf{x} = (x_1,\ldots, x_{\lvert V_j \rvert})$ with coefficients in $\C$ (or any field of characteristic $0$); see \cite[Ch.~I, (2.12)]{MacdonaldSymBook}.
With respect to the above action of $\mathcal{S}_{\lvert V_j \rvert}$ on $\C^{\lvert V_j \rvert}$,
the symmetric polynomials are exactly the \emph{invariant polynomials}, i.e., the polynomials $f(\mathbf{x})$ on $\C^{\lvert V_j \rvert}$
such that $f(g\mathbf{x}) = f(\mathbf{x})$ for all $g\in\mathcal{S}_{\lvert V_j \rvert}$. As $\mathcal{S}_{\lvert V_j \rvert}$ is finite, the invariant
polynomials are known to separate points in the set of $\mathcal{S}_{\lvert V_j \rvert}$-orbits
of points in $\C^{\lvert V_j \rvert}$; see \cite[Corollary~2.3.8]{DerksenKemper}. That is, if $\mathbf{x}, \mathbf{y}\in\C^{\lvert V_j \rvert}$ such that
$g\mathbf{x} \neq \mathbf{y}$ for all $g\in \mathcal{S}_{\lvert V_j \rvert}$, then there is a symmetric polynomial $f$ such that
$f(\mathbf{x})\neq f(\mathbf{y})$. Hence, as each such symmetric polynomial $f$ can be written as a polynomial of the first $\lvert V_j \rvert$
power sums, the values of the first $\lvert V_j \rvert$ power sums of the $\lambda_{j,i}$ determine the
$\mathcal{S}_{\lvert V_j \rvert}$-orbit $\mathcal{S}_{\lvert V_j \rvert}\boldsymbol{\lambda}_j$, and therefore determine
the multiset of eigenvalues $\{\lambda_{j,1},\ldots,\lambda_{j,\lvert V_j \rvert}\} = \aSpec(D_j)$. Therefore, if $\lvert V_1\rvert=\lvert V_2\rvert$ and
$N_m(D_1)=N_m(D_2)$ for all $m \leq \lvert V_1\rvert$, it follows that $\aSpec(D_1) = \aSpec(D_2)$. If $|V_1| < |V_2|$ and $N_m(D_1)=N_m(D_2)$ for all
$m \leq \lvert V_2\rvert$, then we may identify $\boldsymbol{\lambda}_1$ with a point in $\C^{\lvert V_2 \rvert}$ by extending by $0$.
Then $\aSpec(D_2)$ is given by $\aSpec(D_1)$ with $\lvert V_2\rvert-\lvert V_1\rvert$ additional zero elements.
\end{proof}

We can now apply Proposition~\ref{prop:BowenLanford} to show that the non-zero adjacency spectrum $\aSpec(D)$ of a digraph $D=(V,E,r,s)$ is equal to the
non-zero line adjacency spectrum $\eSpec(D)$; recall that the latter is defined to be the adjacency spectrum of the line digraph $\ed(D)$ of $D$. Because
the adjacency matrices of $D$ and $\ed(D)$ are generally not the same size, the spectra are not exactly equal in general. After removing zeros, however,
the spectra do match. Alternate proofs of this result can be found in \cite[Theorem~1.2]{Brualdi}, \cite[Theorem~1.4.4]{LiuLai}, and the references contained
therein.

\begin{proposition}[\cite{Brualdi,LiuLai}]
\label{prop:edgedualspectrum}
Let $D=(V,E,r,s)$ be a finite digraph.  The multiset of non-zero elements of $\aSpec(D)$ is equal to the multiset of non-zero elements of $\aSpec(\ed(D))$.
\end{proposition}

\begin{proof}
By Corollary~\ref{cor:CycleCountInfinite} or \ref{cor:CycleCountFinite}, it suffices to show that $N_m(D)=N_m(\ed(D))$ for all $m\in \N$. If
$C=(e_1,e_2,\dots,e_m)$ is a cycle in $D$, then $C$ corresponds to a cycle $C^{\prime}$ of length $m$ in $\ed(D)$ given by edges $\big((e_1,e_2),(e_2,e_3),\ldots,(e_{m-1},e_m),(e_m,e_1)\big)$ in $\ed(D)$. By the definition of the line digraph, the
correspondence $C\mapsto C^{\prime}$ is a bijection.
\end{proof}

\section{Morita equivalence moves and spectrum}
\label{sec:MovesMorita}

For a digraph $D=(V,E,r,s)$, we consider transformations of the graph, referred to as Moves (S), (R), (I), (O), (C), and (P), following \cite{BatesPask, EilersRestorffEA17, EilersRestorffEA, Sorensen}.  It was shown in \cite{Sorensen} that if (finite) digraphs $D_1$ and $D_2$ differ by a sequence of Moves (S), (R), (O), and (I), then the associated $C^*$-algebras are stably equivalent, and thus Morita equivalent.  In \cite{EilersRestorffEA17}, this list of moves was extended when it was shown that if $D_1$ and $D_2$ differ by Move (C), then they are stably equivalent.  Finally, in \cite{EilersRestorffEA}, by introducing Move (P) the authors achieved a full classification: digraphs $D_1$ and $D_2$ having finitely many vertices have Morita equivalent $C^*$-algebras if and only if they differ by a finite sequence of Moves (S), (R), (O), (I), (C), and (P) and their inverses; see \cite[Theorem 3.1]{EilersRestorffEA}.

For the ease of reading, we describe the moves here.  Following this, we apply Corollaries~\ref{cor:CycleCountInfinite} and \ref{cor:CycleCountFinite} to determine which of the moves preserve the non-zero portion of the adjacency spectrum.  In Section~\ref{sec:NegResults}, we give a series of counterexamples that show which of the spectra of Definition~\ref{def:spectra} are \emph{not} preserved by the various moves.

\subsection{Move (S), remove a regular source}

(See \cite[Section 3]{Sorensen}.)
Let $D = (V, E, r , s)$  be a digraph, and  let $v_s \in V$
be  a regular vertex that is a source.
Define a digraph $D^{(S)}=(V^{(S)},E^{(S)},r^{(S)},s^{(S)})$ by
\[ V^{(S)} := V\backslash \{v_s\}, \ E^{(S)} : =E \backslash  \{ s^{-1} (v_s) \}  ,\
 r^{(S)} :=r|_{E^{(S)}},\  s^{(S)} :=s|_{E^{(S)}}. \]

\subsection{Move (R), reduce at a regular vertex}

(See \cite[Section 3]{Sorensen}.)
Let $D = (V, E, r, s)$  be a digraph, and  let $v_r \in V$
be  a regular vertex such that   $s^{-1}(v_r) $ and $s(r^{-1}(v_r))$   are one-point sets.
Let $u$ be  the only vertex that emits to $v_r$ and let $ f$  be the only edge $v_r$ emits.
 Define a digraph $D^{(R)}=(V^{(R)},E^{(R)},r^{(R)},s^{(R)})$ by
\[ V^{(R)} := V\backslash \{v_r\} \text{ and }  E^{(R)} : = \Big( E \backslash  \Big(  r^{-1} (v_r)\bigcup \{ f \}  \Big) \Big)  \bigcup\left\{ ef\, |\,   e \in r^{-1} (v_r) \right\}, \]
 with range and source maps that extend those of $D$ and satisfy $r^{(R)}(ef)=r(f)$  and $s^{(R)}(ef)=s(e)=u$.

\subsection{Move (O), outsplit at a non-sink}\label{sub:Move-O}

(See \cite[Section 3]{BatesPask} and \cite[Section 3]{Sorensen}.)
Let $D = (V, E, r, s)$ be a digraph, and let $v \in V$.
Suppose that $v$ is not a sink.
Partition $s^{-1}(v)$ into a finite number, $n$ say, of sets $\{ \mathcal{E}_1, \mathcal{E}_2, \ldots, \mathcal{E}_n \}$.
Define a digraph $D^{(O)}=(V^{(O)},E^{(O)},r^{(O)},s^{(O)})$ by:
\begin{align*}
 	V^{(O)} & = (V \setminus \{ v \}) \cup \{ v^1, v^2, \ldots, v^n \} \text{ and }\\
	E^{(O)} & = \left( E \setminus r^{-1}(v) \right) \cup \{ e^1, e^2, \ldots, e^n \mid e \in E, r(e) = v \}.
\end{align*}
For $e \notin r^{-1}(v)$ let $r^{(O)}(e) = r(e)$
and for $e \in r^{-1}(v)$ let $r^{(O)}(e^i) = v^i$ for $i = 1, 2, \ldots,n$.
For $e \notin s^{-1}(v)\cup r^{-1}(v)$ let $s^{(O)}(e) = s(e)$;
for $e \in r^{-1}(v)\setminus s^{-1}(v)$ let $s^{(O)}(e^j) = s(e)$ for $j=1,\ldots,n$;
for $e \in s^{-1}(v) \setminus r^{-1}(v)$ let $s^{(O)}(e) = v^i$ if $e \in \mathcal{E}_i$;
and for $e \in s^{-1}(v) \cap r^{-1}(v)$ let $s^{(O)}(e^j) = v^i$ for $j= 1,2, \ldots, n$ if $e \in \mathcal{E}_i$.

\subsection{Move (I), insplit at a regular non-source}\label{sub:Move-I}

(See \cite[Section 5]{BatesPask} and \cite[Section 3]{Sorensen}.)
Let $D = (V, E, r,s)$ be a digraph, and let $v \in V$.
Suppose that $v$ is not a source.
Partition $r^{-1}(v)$ into a finite number, $n$ say, of sets $\{ \mathcal{E}_1, \mathcal{E}_2, \ldots, \mathcal{E}_n \}$.
Define a digraph $D^{(I)}=(V^{(I)},E^{(I)},r^{(I)},s^{(I)})$ by:
\begin{align*}
 	V^{(I)} & = (V \setminus \{ v \}) \cup \{ v^1, v^2, \ldots, v^n \} \text{ and } \\
	E^{(I)} & = \left( E\setminus s^{-1}(v) \right) \cup \{ e^1, e^2, \ldots, e^n \mid e \in E, s(e) = v \}.
\end{align*}
For $e \notin s^{-1}(v)$ let $s^{(I)}(e) = s(e)$
and for $e \in s^{-1}(v)$ let $s^{(I)}(e^i) = v^i$ for $i= 1,2, \ldots, n$.
For $e \notin r^{-1}(v)\cup s^{-1}(v)$ let $r^{(I)}(e) = r(e)$;
for $e \in s^{-1}(v) \setminus r^{-1}(v)$  let $r^{(I)}(e^j) = r(e)$ for $j=1,\ldots,n$;
for $e \in r^{-1}(v) \setminus s^{-1}(v)$  let $r^{(I)}(e) = v^i$ if $e \in \mathcal{E}_i$;
and for $e \in r^{-1}(v) \cap s^{-1}(v)$ let $r^{(I)}(e^j)=v^i$ for $j=1,2,\ldots, n$ if $e \in \mathcal{E}_i$.

\subsection{Move (C), the Cuntz splice at a vertex admitting at least two distinct return paths}
\label{sub:Move-C}

(See \cite[Section 2]{EilersRestorffEA17} and \cite[Section 3]{Sorensen}.)
Let $D=(V,E,r,s)$ be a digraph, and let $v\in V$ be a regular vertex that supports at least two return paths.  Define a digraph $D^{(C)}=(V^{(C,v)},E^{(C,v)},r^{(C,v)},s^{(C,v)})$ by
\begin{align*}
V^{(C,v)} &= V \cup \{ u_1, u_2 \} \text{ and }\\
E^{(C,v)} &= E \cup \{ e_1, e_2, f_1, f_2, h_1, h_2\},
\end{align*}
where $r^{(C,v)}|_E=r$, $s^{(C,v)}|_E=s$,
\[
	r^{(C,v)}(e_1) = u_1, \quad r^{(C,v)}(e_2) = v, \quad r^{(C,v)}(f_i) = u_i, \quad r^{(C,v)}(h_i) = u_i,
\]
and
\[
	s^{(C,v)}(e_1) = v, \quad s^{(C,v)}(e_2) = u_1, \quad s^{(C,v)}(f_i) = u_1, \quad s^{(C,v)}(h_i) = u_2.
\]
See Figure~\ref{fig:GenericC}.
We say that $D^{(C,v)}$ is the digraph obtained from $D$ by performing Move (C) at vertex $v$.

If $S$ is a subset of $E$ such that each $w\in S$ is a regular vertex supporting at least two return paths, then we may perform Move (C) at each $w\in S$.  We label the resulting digraph $D^{(C,S)}$, and in this case refer to the additional vertices by $u_i^w$, $i=1,2$, for each $w\in S$.

\begin{remark}\label{rem:returnpath}
We note that the Cuntz move can in theory be performed at any vertex, regardless of whether it has at least two distinct return paths.  The condition on return paths is required in order to conclude that the corresponding $C^*$-algebras are Morita equivalent; see \cite[Section 4]{EilersRestorffEA17}.
\end{remark}

\subsection{Move (P), eclose a cyclic component}\label{sub:Move-P}

(See \cite[Definition~2.4]{EilersRestorffEA}.)
Let $D=(V,E,r,s)$ be a digraph and let $v\in V$ support a loop but no other return path.  Suppose that the loop has an exit.  Note that these conditions imply that there is at least one vertex $w\in V$ distinct from $v$ such that there is an edge $e$ with $s(e)=v$ and $r(e)=w$.

Let $S=\{w\in V\backslash\{v\}\,\vert\, s^{-1}(v)\cap r^{-1}(w)\neq \emptyset\}$.  Suppose further that if $w\in S$, then $w$ is a regular vertex that supports at least two return paths. Recall from Section~\ref{sub:Move-C} that $V^{(C,S)}$ is the vertex set of the graph $D^{(C,S)}$ formed by applying Move (C) at each
$w\in S$ and the additional vertices are denoted $u_i^w$ for $w\in S$ and $i=1,2$.
Define a digraph $D^{(P,v)}=(V^{(P,v)},E^{(P,v)},r^{(P,v)},s^{(P,v)})$ by
\begin{align*}
V^{(P,v)} &= V^{(C,S)}\\
E^{(P,v)} &= E^{(C,S)}\cup \{ \bar{e}_w,\tilde{e}_w\,\vert\, w\in S, e\in s^{-1}(v)\cap r^{-1}(w)\},
\end{align*}
where $r^{(P,v)}|_{E^{(C,S)}}=r^{(C,S)}$, $s^{(P,v)}|_{E^{(C,S)}}=s^{(C,S)}$, $r^{(P,v)}(\bar{e}_w)=r^{(P,v)}(\tilde{e}_w)=u_{2}^w$, and $s^{(P,v)}(\bar{e}_w)=s^{(P,v)}(\tilde{e}_w)=v$.

We say that $D^{(P,v)}$ is the digraph formed by performing Move (P) at $v$.

As in Remark~\ref{rem:returnpath}, Move (P) can in theory be performed at any vertex.  The requirement in the definition of Move (P) that vertex $w$ be a regular vertex that supports at least two return paths ensures that when two digraphs differ by Move (P), their corresponding $C^*$-algebras are Morita equivalent.
See \cite[Section 2.5]{EilersRestorffEA}.

\subsection{Moves that preserve spectra}
\label{sub:MovePres}

Based on these descriptions of the moves given above, we can now prove the positive part of Theorem~\ref{thrm:MainResult},
which we organize into four propositions. We begin with the following.

\begin{proposition}
\label{prop:movesadjacencyspectrum}
Given a finite digraph $D$, let $D^{(S)}$, $D^{(O)}$, and $D^{(I)}$ be digraphs resulting from performing Move (S), (O), and (I) to $D$ respectively.  Then the multisets of nonzero elements in $\aSpec(D)$, $\aSpec(D^{(S)})$, $\aSpec(D^{(O)})$, and $\aSpec(D^{(I)})$ are all equal.
\end{proposition}

\begin{proof}
The result follows directly from either Corollary~\ref{cor:CycleCountInfinite} or \ref{cor:CycleCountFinite}.

For Move (S), since the vertex that is removed is a source, it is not part of any cycles, and therefore the number of cycles of a given length in $D$ is equal to the number of cycles of that length in $D^{(S)}$.

Now, suppose that $D^{(O)}$ is the digraph that results from performing Move (O) to $D$ at vertex $v$.  Then there is a bijection between the cycles of length $m$ in $D$ and the cycles of length $m$ in $D^{(O)}$ as follows.  Suppose that $e_1,e_2,\dots, e_m$ is a collection of edges in $D$ forming a cycle $(e_1,e_2,\dots,e_m)$. If $s(e_j)\neq v$ for $j=1,\dots, m$, then by definition of Move (O), this cycle is mapped to exactly one cycle, namely $(e_1,e_2,\dots,e_m)$ in $D^{(O)}$.  If $s(e_j) = v$ for some $j=1,\dots,m$, suppose without loss of generality that $s(e_1) = v$. Using notation from above, suppose further that $e_1\in\mathcal{E}_i$.  Then the cycle $(e_1,e_2,\dots,e_m)$ in $D$ is mapped to cycle $(e_1,e_2,\dots,e_m^i)$ in $D^{(O)}$, where $r(e_m^i) = s(e_1) = v^i$. Again, by definition of Move (O), this correspondence is one-to-one.

On the other hand, suppose that $(e_1,e_2,\dots,e_m)$ is a cycle of length $m$ in $D^{(O)}$.  If $s(e_j)\neq v^i$ for any $j=1,\dots,m$, $i=1,\dots,n$, then, as above $(e_1,e_2,\dots,e_m)$ in $D^{(O)}$ is the image under Move (O) of $(e_1,e_2,\dots, e_m)$ in $D$.  If $s(e_j) = v^i$ for some $j=1,\dots,m$, $i=1,\dots,n$, suppose without loss of generality that $s(e_1)=v^i$ for some $i=1,\dots,n$. Then $e_m = e^i$ for some edge $e$ in $D$ with $r(e) = v$, and
$(e_1,e_2,\dots, e_m) = (e_1,\ldots,e_{m-1},e^i)$ in $D^{(O)}$ is the image under Move (O) of cycle $(e_1,e_2,\dots,e_{m-1},e)$ in $D$.   Thus, we have a bijective correspondence between the cycles of length $m$ in $D$ and the cycles of length $m$ in $D^{(O)}$, and by Corollary~\ref{cor:CycleCountInfinite} or \ref{cor:CycleCountFinite}, the multisets of nonzero elements in $\aSpec(D)$ and $\aSpec(D^{(O)})$ are equal.  The argument that the multisets of nonzero elements in $\aSpec(D)$ and $\aSpec(D^{(I)})$ are equal is similar.
\end{proof}

\begin{proposition}
\label{prop:moveslineadjacencyspectrum}
Given a finite digraph $D$, let $D^{(S)}$, $D^{(O)}$, and $D^{(I)}$ be digraphs resulting from performing Move (S), (O), and (I) to $D$ respectively.  Then the multisets of nonzero elements in $\eSpec(D)$, $\eSpec(D^{(S)})$, $\eSpec(D^{(O)})$, and $\eSpec(D^{(I)})$ are all equal.
\end{proposition}

\begin{proof}
The result follows directly from Propositions~\ref{prop:edgedualspectrum} and \ref{prop:movesadjacencyspectrum}.
\end{proof}

\begin{proposition}
\label{prop:movesbinaryadjacencyspectrum}
Given a finite digraph $D$, let $D^{(S)}$ be a digraph resulting from performing Move (S) to $D$.  Then the multisets of nonzero elements in $\baSpec(D)$ and $\baSpec(D^{(S)})$ are equal.
\end{proposition}

\begin{proof}
The binary adjacency spectrum of a graph can be obtained by first replacing all multiple directed edges from vertex $v$ to vertex $w$ with a single directed edge from $v$ to $w$, then computing the spectrum of the adjacency matrix of the resulting digraph.  Because the binary adjacency spectrum is computed using an adjacency matrix, Corollary~\ref{cor:CycleCountInfinite} and \ref{cor:CycleCountFinite} apply; here we need only confirm that the number of cycles of length $m$ in the resulting digraph is equal before and after we perform Move (S).  But since Move (S) is the removal of a source $v_s$, the vertex $v_s$ is not included in any cycles, and thus the number of cycles of length $m$ remains unchanged, as desired.  The result follows.
\end{proof}

\begin{proposition}
\label{prop:movesskewspectra}
Given a finite digraph $D$, let $D^{(C)}$ be a digraph resulting from performing Move (C) to $D$. Then the multisets of nonzero elements in $\sSpec(D)$ and $\sSpec(D^{(C)})$ are equal,
the multisets of nonzero elements in $\bsSpec(D)$ and $\bsSpec(D^{(C)})$ are equal,
the multisets of nonzero elements in $\slSpec(D)$ and $\slSpec(D^{(C)})$ are equal,
and the multisets of nonzero elements in  $\bslSpec(D)$ and $\bslSpec(D^{(C)})$ are equal.
\end{proposition}
\begin{proof}
Following the notation of Section~\ref{sub:Move-C}, let $v$ be the vertex at which Move (C) is applied
and let $V\cup \{u_1,u_2\}$ be the vertex set of $D^{(C)}$. The new vertices and edges that are added
by Move (C) are pictured in Figure~\ref{fig:GenericC}. In particular, note that $\dout_v$, $\din_v$,
$\bdout_v$, and $\bdin_v$ are each increased by $1$. Furthermore, for each $w,w^\prime\in V$,
the $(w,w^\prime)$-entries of $\sam(D)$ and $\sam(D^{(C)})$ coincide, as do the $(w,w^\prime)$-entries of
$\bsam(D)$ and $\bsam(D^{(C)})$, the diagonal entries $\dout_w - \din_w$ of $\slm(D)$ and $\slm(D^{(C)})$,
and the diagonal entries $\bdout_w - \bdin_w$ of $\bslm(D)$ and $\bslm(D^{(C)})$. The entries corresponding
to $u_1$ and $u_2$ are easily seen to vanish so that, ordering the vertex set of $D^{(C)}$ with $u_1, u_2$
as the last two elements, we have
\begin{align*}
    \sam(D^{(C)}) &=
    \left(\begin{array}{ccc|cc}
        &           &&      0       &   0       \\
        & \sam(D)   &&      \vdots  &   \vdots  \\
        &           &&      0       &   0       \\
        \hline
        0 &\cdots& 0 &      0       &   0       \\
        0 &\cdots& 0 &      0       &   0
    \end{array}\right),
    &
    \bsam(D^{(C)}) =
    \left(\begin{array}{ccc|cc}
        &           &&      0       &   0       \\
        & \bsam(D)  &&      \vdots  &   \vdots  \\
        &           &&      0       &   0       \\
        \hline
        0 &\cdots& 0 &      0       &   0       \\
        0 &\cdots& 0 &      0       &   0
    \end{array}\right)&,
\\
    \slm(D^{(C)}) &=
    \left(\begin{array}{ccc|cc}
        &           &&      0       &   0       \\
        & \slm(D)   &&      \vdots  &   \vdots  \\
        &           &&      0       &   0       \\
        \hline
        0 &\cdots& 0 &      0       &   0       \\
        0 &\cdots& 0 &      0       &   0
    \end{array}\right),
    &\text{and}\quad
    \bslm(D^{(C)}) =
    \left(\begin{array}{ccc|cc}
        &           &&      0       &   0       \\
        & \bslm(D)  &&      \vdots  &   \vdots  \\
        &           &&      0       &   0       \\
        \hline
        0 &\cdots& 0 &      0       &   0       \\
        0 &\cdots& 0 &      0       &   0
    \end{array}\right).&
\end{align*}
The claim then follows.
\end{proof}

\begin{figure}
\begin{tikzpicture}[scale=.7]
\node (v) at (0,0){$v$};
\node (u1) at (2,0){$u_1$};
\node (u2) at (4,0){$u_2$};
\node (v1) at (-3,2){};
\node (v2) at (-3.5,1){};
\node (v3) at (-3.5,0){};
\node (v4) at (-3,-2){};
\node (dots) at (-1.3,-.3){$\vdots$};
\node (e1) at (1,.8){$e_1$};
\node (e2) at (1,-.7){$e_2$};
\node (f1) at (2,-1.9){$f_1$};
\node (f2) at (3,.8){$f_2$};
\node (h1) at (3.1,-.8){$h_1$};
\node (h2) at (5.8,0){$h_2$};
\draw[->,black] (v) ..  controls (.7, .5) and (1.3, .5) .. (u1);
\draw[->,black] (u1) ..  controls (1.3, -.5) and (.7, -.5) .. (v);
\draw[->,black] (u1) ..  controls (2.7, .5) and (3.3, .5) .. (u2);
\draw[->,black] (u2) ..  controls (3.3, -.5) and (2.7, -.5) .. (u1);
\draw[->,black] (u1) ..  controls (1, -2) and (3, -2) .. (u1);
\draw[->,black] (u2) ..  controls (5.8, -1) and (5.8, 1) .. (u2);
\draw[->, black,dashed] (v1) to  (v);
\draw[->, black,dashed] (v) to  (v2);
\draw[->, black,dashed] (v3) to  (v);
\draw[->, black,dashed] (v) to  (v4);
\end{tikzpicture}
\caption{The vertices and edges added to a digraph by Move (C).}
\label{fig:GenericC}
\end{figure}
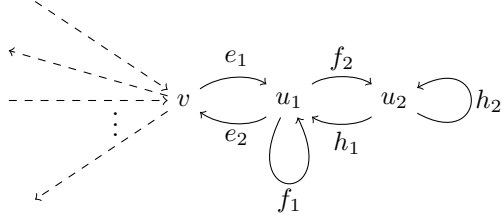

\section{Negative results}
\label{sec:NegResults}

In this section, we give counterexamples to illustrate the failures of the various spectra to be preserved under the remaining moves, completing the proof of Theorem~\ref{thrm:MainResult}. Specifically, Examples~\ref{ex:NegD1},
\ref{ex:NegD2}, \ref{ex:NegD3}, \ref{ex:NegD4}, and \ref{ex:NegD5} are sufficient to indicate the failures of
moves to preserve spectra stated in Theorem~\ref{thrm:MainResult}. Note in particular that the digraphs in
Examples~\ref{ex:NegD1} and \ref{ex:NegD2} are not strongly connected, while the digraphs in
Examples~\ref{ex:NegD3}, \ref{ex:NegD4}, and \ref{ex:NegD5} are strongly connected. To clarify the organization,
Table~\ref{tab:ProofOrgGeneral} indicates which counterexample applies to each spectrum and move pair, while
Table~\ref{tab:ProofOrgStrongConn} indicates counterexamples within the class of strongly connected digraphs.

Note that if digraphs $D_1$ and $D_2$ are related by Move (S), then one must contain a source and therefore is not strongly connected. Similarly, Move (P) only applies to a digraph containing a vertex $u$ that supports a loop and no other return path, and such that the loop at $u$ has an exit. Therefore, the exit begins with an edge from $u$ to another vertex $v$ which cannot begin a return path, implying that there is no path from $v$ to $u$ and that the digraph is not strongly connected. That is, Moves (S) and (P) do not apply within the class of strongly connected digraphs, and the spectra
$\nlSpec$, $\clSpec$, $\bnlSpec$, and $\bclSpec$ are not defined for at least one element of a pair of digraphs connected
by Move (S) or (P).

\begin{table}[h!]
\begin{tabular}{|c|c|c|c|c|c|c|}
  \hline
  \textbf{Spectrum} & \textbf{Move (S)} & \textbf{Move (R)} & \textbf{Move (O)} & \textbf{Move (I)} & \textbf{Move (C)} &   \textbf{Move (P)}
  \\ \hline
  $\lSpec$
  & Ex.~\ref{ex:NegD1} & Ex.~\ref{ex:NegD1} & Ex.~\ref{ex:NegD1} & Ex.~\ref{ex:NegD1} & Ex.~\ref{ex:NegD1} & Ex.~\ref{ex:NegD2} \\ \hline
  $\aSpec$
  & Prop.~\ref{prop:movesadjacencyspectrum} & Ex.~\ref{ex:NegD1} & Prop.~\ref{prop:movesadjacencyspectrum} & Prop.~\ref{prop:movesadjacencyspectrum} & Ex.~\ref{ex:NegD1} & Ex.~\ref{ex:NegD2} \\ \hline
  $\baSpec$
  & Prop.~\ref{prop:movesbinaryadjacencyspectrum} & Ex.~\ref{ex:NegD1} & Ex.~\ref{ex:NegD3} & Ex.~\ref{ex:NegD1} & Ex.~\ref{ex:NegD1} & Ex.~\ref{ex:NegD2} \\ \hline
  $\aSpecS$
  & Ex.~\ref{ex:NegD1} & Ex.~\ref{ex:NegD1} & Ex.~\ref{ex:NegD1} & Ex.~\ref{ex:NegD1} & Ex.~\ref{ex:NegD1} & Ex.~\ref{ex:NegD2} \\ \hline
  $\baSpecS$
  & Ex.~\ref{ex:NegD1} & Ex.~\ref{ex:NegD1} & Ex.~\ref{ex:NegD1} & Ex.~\ref{ex:NegD1} & Ex.~\ref{ex:NegD1} & Ex.~\ref{ex:NegD2} \\ \hline
  $\eSpec$
  & Prop.~\ref{prop:moveslineadjacencyspectrum} & Ex.~\ref{ex:NegD1} & Prop.~\ref{prop:moveslineadjacencyspectrum} & Prop.~\ref{prop:moveslineadjacencyspectrum} & Ex.~\ref{ex:NegD1} & Ex.~\ref{ex:NegD2} \\ \hline
  $\hSpec$
  & Ex.~\ref{ex:NegD1} & Ex.~\ref{ex:NegD1} & Ex.~\ref{ex:NegD1} & Ex.~\ref{ex:NegD1} & Ex.~\ref{ex:NegD1} & Ex.~\ref{ex:NegD2} \\ \hline
  $\sSpec$
  & Ex.~\ref{ex:NegD1} & Ex.~\ref{ex:NegD5} & Ex.~\ref{ex:NegD1} & Ex.~\ref{ex:NegD1} & Prop.~\ref{prop:movesskewspectra} & Ex.~\ref{ex:NegD2} \\ \hline
  $\bsSpec$
  & Ex.~\ref{ex:NegD1} & Ex.~\ref{ex:NegD5} & Ex.~\ref{ex:NegD1} & Ex.~\ref{ex:NegD1} & Prop.~\ref{prop:movesskewspectra} & Ex.~\ref{ex:NegD2} \\ \hline
  $\slSpec$
  & Ex.~\ref{ex:NegD1} & Ex.~\ref{ex:NegD5} & Ex.~\ref{ex:NegD1} & Ex.~\ref{ex:NegD1} & Prop.~\ref{prop:movesskewspectra} & Ex.~\ref{ex:NegD2} \\ \hline
  $\bslSpec$
  & Ex.~\ref{ex:NegD1} & Ex.~\ref{ex:NegD5} & Ex.~\ref{ex:NegD1} & Ex.~\ref{ex:NegD1} & Prop.~\ref{prop:movesskewspectra} & Ex.~\ref{ex:NegD2} \\ \hline
  $\nlSpec$
  & NA & Ex.~\ref{ex:NegD4} & Ex.~\ref{ex:NegD3} & Ex.~\ref{ex:NegD3} & Ex.~\ref{ex:NegD4} & NA \\ \hline
  $\clSpec$
  & NA & Ex.~\ref{ex:NegD4} & Ex.~\ref{ex:NegD3} & Ex.~\ref{ex:NegD3} & Ex.~\ref{ex:NegD4} & NA \\ \hline
  $\bnlSpec$
  & NA & Ex.~\ref{ex:NegD4} & Ex.~\ref{ex:NegD3} & Ex.~\ref{ex:NegD3} & Ex.~\ref{ex:NegD4} & NA \\ \hline
  $\bclSpec$
  & NA & Ex.~\ref{ex:NegD4} & Ex.~\ref{ex:NegD3} & Ex.~\ref{ex:NegD3} & Ex.~\ref{ex:NegD4} & NA \\ \hline
\end{tabular}
\caption{For each move and spectrum pair, the result demonstrating that the move preserves the nonzero elements
of the spectrum or a counterexample indicating that the nonzero spectral elements are not preserved.
Moves (S) and (P) either must be applied to or produce non-strongly connected digraphs, for which the spectra
$\nlSpec$, $\clSpec$, $\bnlSpec$, and $\bclSpec$ are not defined.}
\label{tab:ProofOrgGeneral}
\end{table}

\begin{table}[h]
\begin{tabular}{|c|c|c|c|c|}
  \hline
  \textbf{Spectrum}  &   \textbf{Move (R)}    &   \textbf{Move (O)}    &   \textbf{Move (I)}    &   \textbf{Move (C)}
  \\ \hline
  $\lSpec$
    & Ex.~\ref{ex:NegD4} & Ex.~\ref{ex:NegD4} & Ex.~\ref{ex:NegD4} & Ex.~\ref{ex:NegD4} \\ \hline
  $\aSpec$
    & Ex.~\ref{ex:NegD4} & Prop.~\ref{prop:movesadjacencyspectrum} & Prop.~\ref{prop:movesadjacencyspectrum} & Ex.~\ref{ex:NegD4} \\ \hline
  $\baSpec$
    & Ex.~\ref{ex:NegD4} & Ex.~\ref{ex:NegD3} & Ex.~\ref{ex:NegD3} & Ex.~\ref{ex:NegD4} \\ \hline
  $\aSpecS$
    & Ex.~\ref{ex:NegD4} & Ex.~\ref{ex:NegD4} & Ex.~\ref{ex:NegD4} & Ex.~\ref{ex:NegD4} \\ \hline
  $\baSpecS$
    & Ex.~\ref{ex:NegD4} & Ex.~\ref{ex:NegD4} & Ex.~\ref{ex:NegD4} & Ex.~\ref{ex:NegD4} \\ \hline
  $\eSpec$
    & Ex.~\ref{ex:NegD4} & Prop.~\ref{prop:moveslineadjacencyspectrum} & Prop.~\ref{prop:moveslineadjacencyspectrum} & Ex.~\ref{ex:NegD4} \\ \hline
  $\hSpec$
    & Ex.~\ref{ex:NegD4} & Ex.~\ref{ex:NegD4} & Ex.~\ref{ex:NegD4} & Ex.~\ref{ex:NegD4} \\ \hline
  $\sSpec$
    & Ex.~\ref{ex:NegD5} & Ex.~\ref{ex:NegD4} & Ex.~\ref{ex:NegD4} & Prop.~\ref{prop:movesskewspectra} \\ \hline
  $\bsSpec$
    & Ex.~\ref{ex:NegD5} & Ex.~\ref{ex:NegD4} & Ex.~\ref{ex:NegD4} & Prop.~\ref{prop:movesskewspectra} \\ \hline
  $\slSpec$
    & Ex.~\ref{ex:NegD5} & Ex.~\ref{ex:NegD4} & Ex.~\ref{ex:NegD4} & Prop.~\ref{prop:movesskewspectra} \\ \hline
  $\bslSpec$
    & Ex.~\ref{ex:NegD5} & Ex.~\ref{ex:NegD4} & Ex.~\ref{ex:NegD4} & Prop.~\ref{prop:movesskewspectra} \\ \hline
  $\nlSpec$
    & Ex.~\ref{ex:NegD4} & Ex.~\ref{ex:NegD4} & Ex.~\ref{ex:NegD4} & Ex.~\ref{ex:NegD4} \\ \hline
  $\clSpec$
    & Ex.~\ref{ex:NegD4} & Ex.~\ref{ex:NegD4} & Ex.~\ref{ex:NegD4} & Ex.~\ref{ex:NegD4} \\ \hline
  $\bnlSpec$
    & Ex.~\ref{ex:NegD4} & Ex.~\ref{ex:NegD4} & Ex.~\ref{ex:NegD4} & Ex.~\ref{ex:NegD4} \\ \hline
  $\bclSpec$
    & Ex.~\ref{ex:NegD4} & Ex.~\ref{ex:NegD4} & Ex.~\ref{ex:NegD4} & Ex.~\ref{ex:NegD4} \\ \hline
\end{tabular}
\caption{For each move and spectrum pair, the result demonstrating that the move preserves the nonzero elements
of the spectrum or a counterexample indicating that the nonzero spectral elements are not preserved within the
class of strongly connected digraphs. Moves (S) and (P) are omitted, as they do not preserve strongly connectedness.}
\label{tab:ProofOrgStrongConn}
\end{table}

\begin{example}
\label{ex:NegD1}
Let $D_1$ denote the digraph of \cite[Fig.~1(a)]{EilersRestorffEA} (there denoted $E$); see
Figure~\ref{subfig:D1}.
Let $D_1^{(S)}$ denote the digraph obtained from $D_1$ by applying the inverse
of Move (S), adjoining a source $v_s$ that has edges to $v_1$, $v_2$, and $v_3$ (Figure~\ref{subfig:D1S});
let $D_1^{(R)}$ denote the digraph obtained by applying the inverse of Move (R) to $D_1$,
adjoining a vertex $v_r$ with edges to and from $v_2$ (see Figure~\ref{subfig:D1R});
let $D_1^{(O)}$ denote be obtained by applying Move (O) at $v_1$ using the partition with two elements,
the first containing the loop at $v_1$ and edge from $v_1$ to $v_2$, the second containing both edges from
$v_1$ to $v_3$ (see Figure~\ref{subfig:D1O}); let $D_1^{(I)}$ denote the digraph obtained by applying
Move (I) at $v_2$ using the partition with two elements, the first containing one loop at $v_2$ and the edge
from $v_1$ to $v_2$, the second containing the other loop at $v_2$ (Figure~\ref{subfig:D1I}); and
let $D_1^{(C)}$ denote the digraph obtained from $D_1$ by applying Move (C) to $D_1$ at vertex $v_2$
(Figure~\ref{subfig:D1C}).
Note that $D_1$ does not satisfy the hypotheses required to apply Move (P) as discussed in
Section~\ref{sub:Move-P}; see Remark~\ref{rem:returnpath}. Specifically, $v_3$ is the range of edges
with source $v_1$ and $v_2$, respectively, and $v_3$ supports only one return path, so Move (P)
cannot be applied at $v_1$ nor $v_2$; since there are no edges from $v_3$ to any other vertex, the
move cannot be applied at $v_3$. Hence, Move (P) will be considered in Example~\ref{ex:NegD2} below.

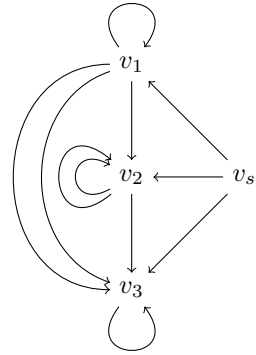
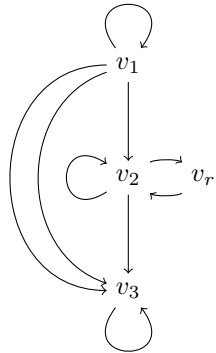
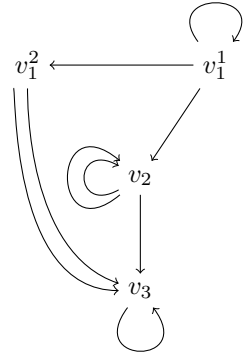
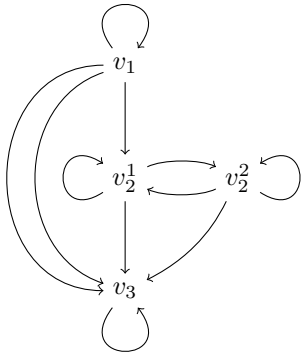
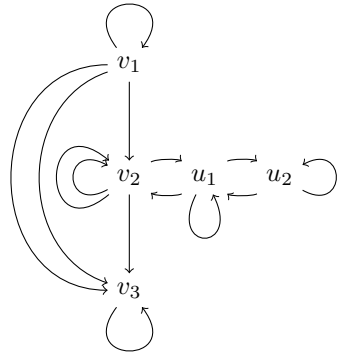
\begin{figure}
\begin{subfigure}{.4\textwidth}\centering
\begin{tikzpicture}[scale=.5]
\node (v1) at (0,3){$v_1$};
\node (v2) at (0,0){$v_2$};
\node (v3) at (0,-3){$v_3$};
\draw[->,black] (v1) ..  controls (-1.5, 5) and (1.5, 5) .. (v1);
\draw[->, black] (v1) to  (v2);
\draw[->,black] (v2) .. controls (-1.8, -1) and (-1.8, 1) .. (v2);
\draw[->,black] (v2) ..  controls (-2.4, -2) and (-2.4, 2) .. (v2);
\draw[->, black] (v1)..  controls (-4, 3) and (-4, -3).. (v3);
\draw[->, black] (v1)..  controls (-3, 2) and (-3, -2).. (v3);
\draw[->, black] (v2) to (v3);
\draw[->,black] (v3) .. controls (-1.5, -5) and (1.5, -5) .. (v3);
\end{tikzpicture}
\caption{The digraph $D_1$.}
\label{subfig:D1}
\end{subfigure}
\hspace{.5cm}
\begin{subfigure}{.4\textwidth}\centering
\begin{tikzpicture}[scale=.5]
\node (v1) at (0,3){$v_1$};
\node (v2) at (0,0){$v_2$};
\node (v3) at (0,-3){$v_3$};
\node (s) at (3,0){$v_s$};
\draw[->,black] (v1) ..  controls (-1.5, 5) and (1.5, 5) .. (v1);
\draw[->, black] (v1) to  (v2);
\draw[->,black] (v2) .. controls (-1.8, -1) and (-1.8, 1) .. (v2);
\draw[->,black] (v2) ..  controls (-2.4, -2) and (-2.4, 2) .. (v2);
\draw[->, black] (v1)..  controls (-4, 3) and (-4, -3).. (v3);
\draw[->, black] (v1)..  controls (-3, 2) and (-3, -2).. (v3);
\draw[->, black] (v2) to (v3);
\draw[->,black] (v3) .. controls (-1.5, -5) and (1.5, -5) .. (v3);
\draw[->,black] (s) to (v1);
\draw[->,black] (s) to (v2);
\draw[->,black] (s) to (v3);
\end{tikzpicture}
\caption{The digraph $D_1^{(S)}$ obtained by applying the inverse of Move (S) to $D_1$,
adjoining source $v_s$.}
\label{subfig:D1S}
\end{subfigure}
\\
\begin{subfigure}{.4\textwidth}\centering
\begin{tikzpicture}[scale=.5]
\node (v1) at (0,3){$v_1$};
\node (v2) at (0,0){$v_2$};
\node (v3) at (0,-3){$v_3$};
\node (r) at (2,0){$v_r$};
\draw[->,black] (v1) ..  controls (-1.5, 5) and (1.5, 5) .. (v1);
\draw[->, black] (v1) to  (v2);
\draw[->,black] (v2) .. controls (-2, -1.3) and (-2, 1.3) .. (v2);
\draw[->, black] (v1)..  controls (-4, 3) and (-4, -3).. (v3);
\draw[->, black] (v1)..  controls (-3, 2) and (-3, -2).. (v3);
\draw[->, black] (v2) to (v3);
\draw[->,black] (v3) .. controls (-1.5, -5) and (1.5, -5) .. (v3);
\draw[->,black] (v2) ..  controls (.7, .5) and (1.3, .5) .. (r);
\draw[->,black] (r) ..  controls (1.3, -.5) and (.7, -.5) .. (v2);
\end{tikzpicture}
\caption{The digraph $D_1^{(R)}$ obtained by applying the inverse of Move (R) to $D_1$,
adjoining the vertex $v_r$.}
\label{subfig:D1R}
\end{subfigure}
\hspace{.5cm}
\begin{subfigure}{.4\textwidth}\centering
\begin{tikzpicture}[scale=.5]
\node (o1) at (2,3){$v_1^1$};
\node (o2) at (-3,3){$v_1^2$};
\node (v2) at (0,0){$v_2$};
\node (v3) at (0,-3){$v_3$};
\draw[->,black] (o1) ..  controls (0.5, 5) and (3.5, 5) .. (o1);
\draw[->, black] (o1) to  (v2);
\draw[->, black] (o1) to  (o2);
\draw[->,black] (v2) .. controls (-1.8, -1) and (-1.8, 1) .. (v2);
\draw[->,black] (v2) ..  controls (-2.4, -2) and (-2.4, 2) .. (v2);
\draw[->, black] (o2)..  controls (-3, 2) and (-3, -2).. (v3);
\draw[->, black] (o2)..  controls (-3.3, 2.5) and (-3.5, -3).. (v3);
\draw[->, black] (v2) to (v3);
\draw[->,black] (v3) .. controls (-1.5, -5) and (1.5, -5) .. (v3);
\end{tikzpicture}
\caption{The digraph $D_1^{(O)}$ obtained by applying Move (O) to $D_1$ at $v_1$ using the partition
whose first element contains the loop at $v_1$ and edge from $v_1$ to $v_2$ and second element contains
both edges from $v_1$ to $v_3$.}
\label{subfig:D1O}
\end{subfigure}
\\
\begin{subfigure}{.4\textwidth}\centering
\begin{tikzpicture}[scale=.5]
\node (v1) at (0,3){$v_1$};
\node (v3) at (0,-3){$v_3$};
\node (i1) at (0,0){$v_2^1$};
\node (i2) at (3,0){$v_2^2$};
\draw[->,black] (v1) ..  controls (-1.5, 5) and (1.5, 5) .. (v1);
\draw[->, black] (v1) to  (i1);
\draw[->,black] (i1) .. controls (-2, -1.3) and (-2, 1.3) .. (i1);
\draw[->, black] (v1)..  controls (-4, 3) and (-4, -3).. (v3);
\draw[->, black] (v1)..  controls (-3, 2) and (-3, -2).. (v3);
\draw[->, black] (i1) to (v3);
\draw[->,black] (v3) .. controls (-1.5, -5) and (1.5, -5) .. (v3);
\draw[->,black] (i1) ..  controls (1, .5) and (2, .5) .. (i2);
\draw[->,black] (i2) ..  controls (2.5, -1) and (2, -2) .. (v3);
\draw[->,black] (i2) ..  controls (2, -.5) and (1, -.5) .. (i1);
\draw[->,black] (i2) ..  controls (5, -1.3) and (5, 1.3) .. (i2);
\end{tikzpicture}
\caption{The digraph $D_1^{(I)}$ obtained by applying Move (I) to $D_1$ at $v_2$ using the partition whose
first element contains one loop at $v_2$ and the edge from $v_1$ to $v_2$ and second contains the other loop at $v_2$.}
\label{subfig:D1I}
\end{subfigure}
\hspace{.5cm}
\begin{subfigure}{.4\textwidth}\centering
\begin{tikzpicture}[scale=.5]
\node (v1) at (0,3){$v_1$};
\node (v2) at (0,0){$v_2$};
\node (v3) at (0,-3){$v_3$};
\node (w1) at (2,0){$u_1$};
\node (w2) at (4,0){$u_2$};
\draw[->,black] (v1) ..  controls (-1.5, 5) and (1.5, 5) .. (v1);
\draw[->, black] (v1) to  (v2);
\draw[->,black] (v2) .. controls (-1.8, -1) and (-1.8, 1) .. (v2);
\draw[->,black] (v2) ..  controls (-2.4, -2) and (-2.4, 2) .. (v2);
\draw[->, black] (v2) to (v3);
\draw[->, black] (v1)..  controls (-4, 3) and (-4, -3).. (v3);
\draw[->, black] (v1)..  controls (-3, 2) and (-3, -2).. (v3);
\draw[->,black] (v3) .. controls (-1.5, -5) and (1.5, -5) .. (v3);
\draw[->,black] (v2) ..  controls (.7, .5) and (1.3, .5) .. (w1);
\draw[->,black] (w1) ..  controls (1.3, -.5) and (.7, -.5) .. (v2);
\draw[->,black] (w1) ..  controls (2.7, .5) and (3.3, .5) .. (w2);
\draw[->,black] (w2) ..  controls (3.3, -.5) and (2.7, -.5) .. (w1);
\draw[->,black] (w1) ..  controls (1, -2) and (3, -2) .. (w1);
\draw[->,black] (w2) ..  controls (5.8, -1) and (5.8, 1) .. (w2);
\end{tikzpicture}
\caption{The digraph $D_1^{(C)}$ obtained by applying Move (C) to $D_1$ at vertex $v_2$.}
\label{subfig:D1C}
\end{subfigure}
\caption{The digraph $D_1$ and the application of Moves (S), (R), (O), (I), and (C) from Example~\ref{ex:NegD1}.}
\label{fig:D1Counterex}
\end{figure}

The Laplace spectra of the digraphs formed from $D_1$ are as follows, with numerical approximations given when
expressions by radicals are cumbersome:
\begin{align*}
    \lSpec(D_1)
        &=          \{ 5, 3, 0 \},
    \\
    \lSpec(D_1^{(S)})
        &=          \{6, 4, 4, 0\},
    \\
    \lSpec(D_1^{(R)})
        &=          \left\{\frac{7+\sqrt{17}}{2}, 5, \frac{7-\sqrt{17}}{2} ,0\right\},
    \\
    \lSpec(D_1^{(O)})
        &=          \left\{4 + \sqrt{2}, 4 - \sqrt{2}, 2, 0 \right\},
    \\
    \lSpec(D_1^{(I)})
        &=          \left\{6,4 + \sqrt{2}, 4 - \sqrt{2}, 0 \right\},
    \\
    \lSpec(D_1^{(C)})
        &=          \{\approx 6.6262, 5, \approx 3.51514, \approx 0.858664, 0\},
\end{align*}
Hence, these examples illustrate that the Laplace spectrum, as well as the submultiset of nonzero
elements, is not preserved by any of the moves (S), (R), (O), (I), nor (C).

The adjacency spectra of the digraphs formed from $D_1$ are:
\begin{align*}
    \aSpec(D_1)
        &=          \{ 2,1,1 \},
    \\
    \aSpec(D_1^{(S)})
        =           \aSpec(D_1^{(O)})   =   \aSpec(D_1^{(I)})    &=     \{ 2,1,1,0 \},
    \\
    \aSpec(D_1^{(R)})
        &=          \left\{ \frac{1 + \sqrt{5}}{2}, 1, 1, \frac{1 - \sqrt{5}}{2} \right\},
    \\
    \aSpec(D_1^{(C)})
        &=          \{\approx 2.80194, \approx 1.44504, 1, 1, \approx -0.24698\}.
\end{align*}
Hence, the submultiset of nonzero elements of the adjacency spectrum is not invariant under moves
(R) and (C) and, as expected from Proposition~\ref{prop:movesadjacencyspectrum}, is not changed by
moves (S), (O), and (I).

The binary adjacency spectra of the digraphs formed from $D_1$ are:
\begin{align*}
    \baSpec(D_1)
        &=          \{ 1, 1, 1 \},
    \\
    \baSpec(D_1^{(S)})
        =   \baSpec(D_1^{(O)})  &=  \{ 1, 1, 1, 0 \}
    \\
    \baSpec(D_1^{(R)})
        &=          \left\{ \frac{1 + \sqrt{5}}{2}, 1, 1, \frac{1 - \sqrt{5}}{2} \right\},
    \\
    \baSpec(D_1^{(I)})
        &=          \{ 2, 1, 1, 0 \},
    \\
    \baSpec(D_1^{(C)})
        &=          \{ 1 + \sqrt{2}, 1, 1, 1, 1 - \sqrt{2} \}.
\end{align*}
The nonzero elements of the binary adjacency spectrum are not preserved by moves (R), (I), and (C);
Move (O) will be treated in Example~\ref{ex:NegD3} below. Move (S) does not change the nonzero
elements of $\baSpec(D_1)$ as indicated by Proposition~\ref{prop:movesbinaryadjacencyspectrum}.

The symmetric adjacency spectra $\aSpecS$ and symmetric binary adjacency spectra $\baSpecS$ are given below.
In each case, the multiset of nonzero elements of the spectrum are not invariant under any of the moves.
\begin{align*}
    \aSpecS(D_1)
        &=          \{ \approx 10.0494, \approx 1.719, \approx 0.231548 \},
    \\
    \aSpecS(D_1^{(S)})
        &=          \{ \approx 12.7148, \approx 1.74411, \approx 0.541128, 0 \},
    \\
    \aSpecS(D_1^{(R)})
        &=          \{ \approx 8.73968, \approx 1.46182, \approx 0.684079, \approx 0.114421\},
    \\
    \aSpecS(D_1^{(O)})
        &=          \left\{ \frac{9 + \sqrt{41}}{2}, 4, \frac{9 - \sqrt{41}}{2}, 0 \right\},
    \\
    \aSpecS(D_1^{(I)})
        &=          \{ \approx 10.8363, \approx 1.86713, \approx 0.296548, 0 \},
    \\
    \aSpecS(D_1^{(C)})
        &=          \{ \approx 11.5864, \approx 4.6524, \approx 1.42213, \approx 0.294867, \approx 0.0442395 \};
\end{align*}
\begin{align*}
    \baSpecS(D_1)
        &=          \{ \approx 5.04892, \approx 0.643104, \approx 0.307979 \},
    \\
    \baSpecS(D_1^{(S)})
        &=          \{ \approx 7.89167, \approx 0.785825, \approx 0.322504, 0 \},
    \\
    \baSpecS(D_1^{(R)})
        &=          \left\{ 3 + 2\sqrt{2}, 1, 1, 3 - 2\sqrt{2} \right\},
    \\
    \baSpecS(D_1^{(O)})
        &=          \left\{ 4, \frac{3 + \sqrt{5}}{2}, \frac{3 - \sqrt{5}}{2}, 0 \right\},
    \\
    \baSpecS(D_1^{(I)})
        &=          \{ \approx 8.12071, \approx 1.31922, \approx 0.560067, 0 \},
    \\
    \baSpecS(D_1^{(C)})
        &=          \{ \approx 7.19584, \approx 3.35194, \approx 0.844535, \approx 0.511755, \approx 0.0959274 \};
\end{align*}

The line adjacency spectra of the digraphs formed by applying moves to $D_1$ are:
\begin{align*}
    \eSpec(D_1)
        &=          \{2, 1, 1, 0, 0, 0, 0, 0\},
    \\
    \eSpec(D_1^{(S)})
        =           \eSpec(D_1^{(I)})   &=  \{2, 1, 1, 0, 0, 0, 0, 0, 0, 0, 0 \},
    \\
    \eSpec(D_1^{(R)})
        &=          \left\{ \frac{1 + \sqrt{5}}{2}, 1, 1, \frac{1 - \sqrt{5}}{2}, 0, 0, 0, 0, 0 \right\},
    \\
    \eSpec(D_1^{(O)})
        &=          \{2, 1, 1, 0, 0, 0, 0, 0, 0 \},
    \\
    \eSpec(D_1^{(C)})
        &=          \{ \approx 2.80194, \approx 1.44504, 1, 1, \approx -0.24698, 0, 0, 0, 0, 0, 0, 0, 0, 0 \}.
\end{align*}
As guaranteed by Corollaries~\ref{cor:CycleCountInfinite} and \ref{cor:CycleCountFinite}, the nonzero elements of
$\eSpec(D_1)$ coincide with the nonzero elements of $\aSpec(D_1)$. As in the case of $\aSpec(D_1)$, we see that the
multiset of nonzero elements of $\eSpec(D_1)$ is not invariant under moves (C) and (R) and is invariant under the
moves guaranteed by Proposition~\ref{prop:movesadjacencyspectrum}.

The Hermitian adjacency spectra are
\begin{align*}
    \hSpec(D_1)
        &=          \{1 + \sqrt{3}, 1, 1 - \sqrt{3} \},
    \\
    \hSpec(D_1^{(S)})
        &=          \{\approx 3.22001, \approx -1.74108, \approx 1.23136, \approx 0.289713 \},
    \\
    \hSpec(D_1^{(R)})
        &=          \{\approx 2.86081, \approx 1.2541, \approx -1.11491, 0\},
    \\
    \hSpec(D_1^{(O)})
        &=          \{\approx 2.81361, \approx -1.34292, 1, \approx 0.529317\},
    \\
    \hSpec(D_1^{(I)})
        &=          \{\approx 3.34292, \approx 1.47068, \approx -0.813607, 0\},
    \\
    \hSpec(D_1^{(C)})
        &=          \{3, 2, -1, 1, 0\},
\end{align*}
demonstrating that $\hSpec$ is not preserved by any of these moves.

The skew adjacency spectra, binary skew adjacency spectra, skew Laplace spectra, and binary skew Laplace spectra are
\begin{align*}
    \sSpec(D_1)
        &=          \{i\sqrt{6}, -i\sqrt{6}, 0\},
    \\
    \sSpec(D_1^{(S)})
        &=           \{3i, -3i, 0, 0\},
    \\
    \sSpec(D_1^{(R)})
        &=          \{i\sqrt{6}, -i\sqrt{6}, 0, 0 \},
    \\
    \sSpec(D_1^{(O)}) = \sSpec(D_1^{(I)})
        &=          \left\{i \sqrt{\frac{7 + 3 \sqrt{5}}{2}},-i \sqrt{\frac{7 + 3 \sqrt{5}}{2}},i \sqrt{\frac{7 - 3 \sqrt{5}}{2}},-i \sqrt{\frac{7 - 3 \sqrt{5}}{2}}\right\},
    \\
    \sSpec(D_1^{(C)})
        &=          \{i\sqrt{6}, -i\sqrt{6}, 0, 0, 0\};
\end{align*}
\begin{align*}
    \bsSpec(D_1)
        &=          \left\{i \sqrt{3},-i \sqrt{3},0\right\},
    \\
    \bsSpec(D_1^{(S)})
        &=           \left\{i \sqrt{3 + 2 \sqrt{2}},-i \sqrt{3 + 2 \sqrt{2}},i \sqrt{3-2 \sqrt{2}},-i \sqrt{3-2 \sqrt{2}}\right\},
    \\
    \bsSpec(D_1^{(R)})
        &=          \left\{i \sqrt{3},-i \sqrt{3},0,0\right\},
    \\
    \bsSpec(D_1^{(O)})
        &=          \{2 i,-2 i,0,0\},
    \\
    \bsSpec(D_1^{(I)})
        &=          \left\{i \sqrt{2+\sqrt{3}},-i \sqrt{2+\sqrt{3}},i \sqrt{2-\sqrt{3}},-i \sqrt{2-\sqrt{3}}\right\},
    \\
    \bsSpec(D_1^{(C)})
        &=          \left\{i \sqrt{3},-i \sqrt{3},0,0,0\right\};
\end{align*}
\begin{align*}
    \slSpec(D_1)
        &=          \left\{-\sqrt{3},\sqrt{3},0\right\},
    \\
    \slSpec(D_1^{(S)}) = \slSpec(D_1^{(I)})
        &=           \left\{-1-\sqrt{3},2,-1+\sqrt{3},0\right\},
    \\
    \slSpec(D_1^{(R)})
        &=          \left\{-\sqrt{3},\sqrt{3},0,0\right\},
    \\
    \slSpec(D_1^{(O)})
        &=          \{0,0,0,0\},
    \\
    \slSpec(D_1^{(C)})
        &=          \left\{-\sqrt{3},\sqrt{3},0,0,0\right\};
\end{align*}
\begin{align*}
    \bslSpec(D_1)
        &=          \{-1,1,0\},
    \\
    \bslSpec(D_1^{(S)})
        &=           \{-2,2,0,0\},
    \\
    \bslSpec(D_1^{(R)})
        &=          \{-1,1,0,0\},
    \\
    \bslSpec(D_1^{(O)})
        &=          \{0,0,0,0\},
    \\
    \bslSpec(D_1^{(I)})
        &=          \{-2,1,1,0\},
    \\
    \bslSpec(D_1^{(C)})
        &=          \{-1,1,0,0,0\}.
\end{align*}
The nonzero elements of these spectra are therefore not preserved by Moves (S), (O), nor (I).
That they are preserved by Move (C) follows from Proposition~\ref{prop:movesskewspectra}. The case of
Move (R) will be treated in Example~\ref{ex:NegD5} below.
\end{example}

Note in particular that the digraphs $D_1^{(S)}$, $D_1^{(R)}$, $D_1^{(O)}$, and $D_1^{(I)}$ in
Example~\ref{ex:NegD1} each have four vertices. Hence, as at least one of moves (S), (R), (O), and (I) fail
to preserve the nonzero elements of each spectrum under consideration, this example also illustrates the
failure of spectra to be preserved when restricting to digraphs with Morita equivalent $C^\ast$-algebras
and the same number of vertices.

The following demonstrates that none of the spectra are preserved by Move (P).

\begin{example}
\label{ex:NegD2}
Let $D_2$ denote the digraph in \cite[Fig.~1(b)]{EilersRestorffEA} (there denoted $F$) pictured
in Figure~\ref{subfig:D2}, and let $D_2^{(P)}$ denote the digraph obtained by applying Move (P) to
$D_2$ at $v_1$ (Figure~\ref{subfig:D2P}). We compute each of the spectra of $D_2$ and $D_2^{(P)}$ in Definition~\ref{def:spectra}
to demonstrate that none of these spectra (or their nonzero elements) are preserved by Move (P).

\begin{figure}
\begin{subfigure}{.4\textwidth}\centering
\begin{tikzpicture}[scale=.5]
\node (v1) at (0,3){$v_1$};
\node (v2) at (0,0){$v_2$};
\node (v3) at (0,-3){$v_3$};
\draw[->,black] (v1) ..  controls (-1.5, 5) and (1.5, 5) .. (v1);
\draw[->, black] (v1) to  (v2);
\draw[->,black] (v2) .. controls (-1.8, -1) and (-1.8, 1) .. (v2);
\draw[->,black] (v2) ..  controls (-2.4, -2) and (-2.4, 2) .. (v2);
\draw[->, black] (v2) to (v3);
\draw[->,black] (v3) .. controls (-1.5, -5) and (1.5, -5) .. (v3);
\end{tikzpicture}
\caption{The digraph $D_2$.}
\label{subfig:D2}
\end{subfigure}
\hspace{.5cm}
\begin{subfigure}{.4\textwidth}\centering
\begin{tikzpicture}[scale=.5]
\node (v1) at (0,3){$v_1$};
\node (v2) at (0,0){$v_2$};
\node (v3) at (0,-3){$v_3$};
\node (w1) at (3,0){$u_1^{v_2}$};
\node (w2) at (6,0){$u_2^{v_2}$};
\draw[->,black] (v1) ..  controls (-1.5, 5) and (1.5, 5) .. (v1);
\draw[->, black] (v1) to  (v2);
\draw[->,black] (v2) .. controls (-1.8, -1) and (-1.8, 1) .. (v2);
\draw[->,black] (v2) ..  controls (-2.4, -2) and (-2.4, 2) .. (v2);
\draw[->, black] (v2) to (v3);
\draw[->,black] (v3) .. controls (-1.5, -5) and (1.5, -5) .. (v3);
\draw[->,black] (v2) ..  controls (.8, .5) and (1.8, .5) .. (w1);
\draw[->,black] (w1) ..  controls (1.8, -.5) and (.8, -.5) .. (v2);
\draw[->,black] (w1) ..  controls (4.1, .5) and (4.8, .5) .. (w2);
\draw[->,black] (w2) ..  controls (4.8, -.5) and (4.1, -.5) .. (w1);
\draw[->,black] (w1) ..  controls (2, -2) and (4, -2) .. (w1);
\draw[->,black] (w2) ..  controls (7.8, -1) and (7.8, 1) .. (w2);
\draw[->,black] (v1) ..  controls (2, 2.8) and (4, 2.3) .. (w2);
\draw[->,black] (v1) ..  controls (2, 3) and (4.5, 3) .. (w2);
\end{tikzpicture}
\caption{The digraph $D_2^{(P)}$ obtained by applying Move (P) to $D_2$ at vertex $v_1$.}
\label{subfig:D2P}
\end{subfigure}
\caption{The digraph $D_2$ and the application of Move (P) from Example~\ref{ex:NegD2}.}
\label{fig:D2Counterex}
\end{figure}
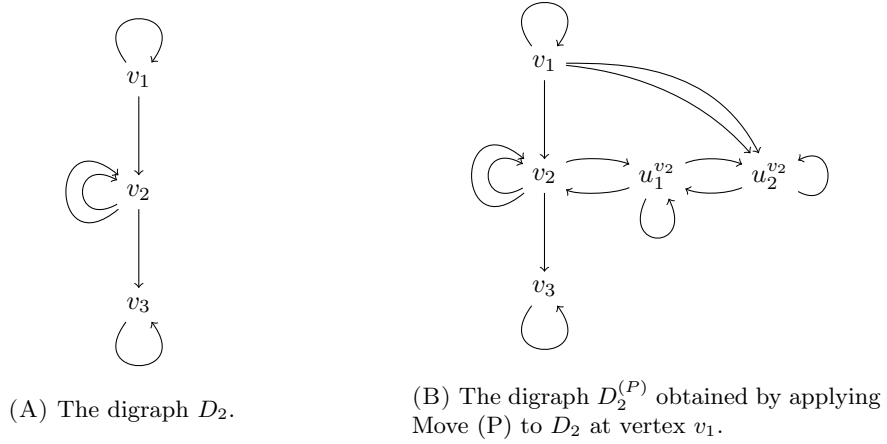

The Laplace spectra are
\begin{align*}
    \lSpec(D_2)
        &=      \{3, 1, 0\},
    \\
    \lSpec(D_2^{(P)})
        &=      \{\approx 7.43874, \approx 4.451, \approx 3.15208, \approx 0.958176, 0 \};
\end{align*}
the adjacency spectra are
\begin{align*}
    \aSpec(D_2)
        &=      \{ 2, 1, 1 \},
    \\
    \aSpec(D_2^{(P)})
        &=      \{ \approx 2.80194, \approx 1.44504, 1, 1, \approx -0.24698 \};
\end{align*}
the binary adjacency spectra are
\begin{align*}
    \baSpec(D_2)
        &=      \{ 1, 1, 1 \},
    \\
    \baSpec(D_2^{(P)})
        &=      \{ 1 + \sqrt{2}, 1, 1, 1, 1 - \sqrt{2} \};
\end{align*}
the symmetric adjacency spectra are
\begin{align*}
    \aSpecS(D_2)
        &=      \{ \approx 6.15633, \approx 1.3691, \approx 0.474572 \},
    \\
    \aSpecS(D_2^{(P)})
        &=      \{ \approx 11.3293, \approx 4.32428, \approx 1.50561 \approx 0.824316, \approx 0.0164465\};
\end{align*}
the symmetric binary adjacency spectra are
\begin{align*}
    \baSpecS(D_2)
        &=      \{ \approx 3.24698 \approx 1.55496, \approx 0.198062 \},
    \\
    \baSpecS(D_2^{(P)})
        &=      \{ \approx 7.45504, \approx 2.40912, \approx 1.52115, \approx 0.547884, \approx 0.0668084 \};
\end{align*}
the line adjacency spectra are
\begin{align*}
    \eSpec(D_2)
        &=      \{2, 1, 1, 0, 0, 0\},
    \\
    \eSpec(D_2^{(P)})
        &=      \{ \approx 2.80194, \approx 1.44504, 1, 1, \approx -0.24698, 0, 0, 0, 0, 0, 0, 0, 0, 0 \};
\end{align*}
the Hermitian adjacency spectra are
\begin{align*}
    \hSpec(D_2)
        &=      \left\{1+\sqrt{2},1,1-\sqrt{2}\right\},
    \\
    \hSpec(D_2^{(P)})
        &=      \Bigg\{\frac{1}{2} \left(2+\sqrt{2 \left(5+\sqrt{17}\right)}\right),
                        \frac{1}{2} \left(2+\sqrt{2 \left(5-\sqrt{17}\right)}\right),
                \\&\quad\quad
                        \frac{1}{2} \left(2-\sqrt{2 \left(5+\sqrt{17}\right)}\right),
                        1,
                        \frac{1}{2} \left(2-\sqrt{2 \left(5-\sqrt{17}\right)}\right)\Bigg\};
\end{align*}
the skew adjacency spectra are
\begin{align*}
    \sSpec(D_2)
        &=      \left\{i \sqrt{2},-i \sqrt{2},0\right\},
    \\
    \sSpec(D_2^{(P)})
        &=      \left\{i \sqrt{\sqrt{5}+3},-i \sqrt{\sqrt{5}+3},i \sqrt{3-\sqrt{5}},-i \sqrt{3-\sqrt{5}},0\right\};
\end{align*}
the binary skew adjacency spectra are
\begin{align*}
    \bsSpec(D_2)
        &=      \left\{i \sqrt{2},-i \sqrt{2},0\right\},
    \\
    \bsSpec(D_2^{(P)})
        &=      \left\{i \sqrt{\frac{1}{2} \left(3+\sqrt{5}\right)},-i \sqrt{\frac{1}{2} \left(3+\sqrt{5}\right)},i \sqrt{\frac{1}{2} \left(3-\sqrt{5}\right)},-i \sqrt{\frac{1}{2} \left(3-\sqrt{5}\right)},0\right\};
\end{align*}
the skew Laplace spectra are
\begin{align*}
    \slSpec(D_2)
        &=      \{i,-i,0\},
    \\
    \slSpec(D_2^{(P)})
        &=      \{-1,1,0,0,0\};
\end{align*}
and the binary skew Laplace spectra are
\begin{align*}
    \bslSpec(D_2)
        &=      \{i,-i,0\},
    \\
    \bslSpec(D_2^{(P)})
        &=      \{0,0,0,0,0\}.
\end{align*}
\end{example}

It is clear that Move (P) will not preserve any of the relevant spectra in general, as this move adds additional vertices and does so in a way that does not merely add zero rows and/or columns to the relevant matrices. It is also evident from Example~\ref{ex:NegD2} that there is no simple relationship between most of the spectra of a digraph before and after applying Move (P). The one exception is the binary adjacency spectra, as in that example,
$\baSpec(D_2) \subset \baSpec(D_2^{(P)})$. As a brief aside, we indicate that this is not the case in general with the following.

\begin{example}
\label{ex:NegD2prime}
Let $D_2^\prime$ denote the digraph pictured in Figure~\ref{subfig:D2prime}, and let $D_2^{\prime(P)}$ denote the digraph obtained by applying Move (P) to
$D_2^\prime$ at $v_1$ (Figure~\ref{subfig:D2primeP}). The binary adjacency spectra of these graphs are
\begin{align*}
    \baSpec(D_2^\prime)
        &=      \{ 2, 1, 1, 0 \},
    \\
    \baSpec(D_2^{\prime(P)})
        &=      \left\{ \frac{3 + \sqrt{5}}{2}, 1 + \sqrt{2}, \frac{1 + \sqrt{5}}{2}, 1, 1, \frac{1 - \sqrt{5}}{2}, 1 - \sqrt{2},
                \frac{3 - \sqrt{5}}{2} \right\},
\end{align*}
indicating that Move (P) does not generally simply concatenate elements to the binary adjacency spectrum.
\begin{figure}
\begin{subfigure}{.4\textwidth}\centering
\begin{tikzpicture}[scale=.5]
\node (w1) at (0,3){$w_1$};
\node (v1) at (0,0){$v_1$};
\node (w2) at (0,-3){$w_2$};
\node (v2) at (-2,3){$v_2$};
\draw[->,black] (w1) ..  controls (-1.5, 5) and (1.5, 5) .. (w1);
\draw[->,black] (w1) ..  controls (-2.4, 5.6) and (2.4, 5.6) .. (w1);
\draw[->, black] (v1) to  (w1);
\draw[->,black] (v1) ..  controls (-2, -1.4) and (-2, 1.4) .. (v1);
\draw[->, black] (v1) to (w2);
\draw[->,black] (w2) .. controls (-1.5, -5) and (1.5, -5) .. (w2);
\draw[->,black] (w2) .. controls (-2.4, -5.6) and (2.4, -5.6) .. (w2);
\draw[->,black] (w1) ..  controls (-0.8,3.3) and (-1.3,3.3) .. (v2);
\draw[->,black] (v2) ..  controls (-1.3,2.7) and (-0.8,2.7).. (w1);
\draw[->,black] (v2) ..  controls (-4, 1.6) and (-4, 4.4) .. (v2);
\end{tikzpicture}
\caption{The digraph $D_2^\prime$.}
\label{subfig:D2prime}
\end{subfigure}
\hspace{.5cm}
\begin{subfigure}{.4\textwidth}\centering
\begin{tikzpicture}[scale=.5]
\node (w1) at (0,3){$w_1$};
\node (v1) at (0,0){$v_1$};
\node (w2) at (0,-3){$w_2$};
\node (v2) at (-3,3){$v_2$};
\node (u11) at (3,3){$u_1^{w_1}$};
\node (u12) at (6,3){$u_2^{w_1}$};
\node (u21) at (3,-3){$u_1^{w_2}$};
\node (u22) at (6,-3){$u_2^{w_2}$};
\draw[->,black] (w1) ..  controls (-1.5, 5) and (1.5, 5) .. (w1);
\draw[->,black] (w1) ..  controls (-2.4, 5.6) and (2.4, 5.6) .. (w1);
\draw[->, black] (v1) to  (w1);
\draw[->,black] (v1) ..  controls (-2, -1.4) and (-2, 1.4) .. (v1);
\draw[->, black] (v1) to (w2);
\draw[->,black] (w2) .. controls (-1.5, -5) and (1.5, -5) .. (w2);
\draw[->,black] (w2) .. controls (-2.4, -5.6) and (2.4, -5.6) .. (w2);
\draw[->,black] (w1) ..  controls (-1,3.3) and (-2,3.3) .. (v2);
\draw[->,black] (v2) ..  controls (-2,2.7) and (-1,2.7).. (w1);
\draw[->,black] (w1) ..  controls (1,3.3) and (2,3.3) .. (u11);
\draw[->,black] (u11) ..  controls (2,2.7) and (1,2.7).. (w1);
\draw[->,black] (u11) ..  controls (4,3.3) and (5,3.3) .. (u12);
\draw[->,black] (u12) ..  controls (5,2.7) and (4,2.7).. (u11);
\draw[->,black] (w2) ..  controls (1,-2.7) and (2,-2.7) .. (u21);
\draw[->,black] (u21) ..  controls (2,-3.3) and (1,-3.3).. (w2);
\draw[->,black] (u21) ..  controls (4,-2.7) and (5,-2.7) .. (u22);
\draw[->,black] (u22) ..  controls (5,-3.3) and (4,-3.3).. (u21);
\draw[->,black] (v2) ..  controls (-5, 1.6) and (-5, 4.4) .. (v2);
\draw[->,black] (u11) ..  controls (1.5, 5) and (4.5, 5) .. (u11);
\draw[->,black] (u12) ..  controls (8.5, 1.6) and (8.5, 4.4) .. (u12);
\draw[->,black] (u21) ..  controls (1.5, -5) and (4.5, -5) .. (u21);
\draw[->,black] (u22) ..  controls (8.5, -1.6) and (8.5, -4.4) .. (u22);
\draw[->,black] (v1) ..  controls (1.8,0.6) and (4.7,2) .. (u12);
\draw[->,black] (v1) ..  controls (2,0.2) and (4.9,1.6) .. (u12);
\draw[->,black] (v1) ..  controls (1.8,-0.6) and (4.7,-2) .. (u22);
\draw[->,black] (v1) ..  controls (2,-0.2) and (4.9,-1.6) .. (u22);
\end{tikzpicture}
\caption{The digraph $D_2^{\prime(P)}$ obtained by applying Move (P) to $D_2^\prime$ at vertex $v_1$.}
\label{subfig:D2primeP}
\end{subfigure}
\caption{The digraph $D_2^\prime$ and the application of Move (P) from Example~\ref{ex:NegD2prime}.}
\label{fig:D2primeCounterex}
\end{figure}
\end{example}

In Example~\ref{ex:NegD1} the nonzero elements of the binary adjacency spectrum were
preserved by Move (O). With the following, we indicate that this does not occur in general and also
consider the spectra $\nlSpec$, $\clSpec$, $\bnlSpec$, and $\bclSpec$ that are only defined for strongly
connected digraphs. We also list the spectra after applying Move (O) and Move (I), as these will be relevant in Example~\ref{ex:NegD4}.

\begin{example}
\label{ex:NegD3}
Let $D_3$ denote the digraph pictured in Figure~\ref{subfig:D3}, and let
$D_3^{(O)}$ denote the digraph obtained from $D_3$ by applying Move (O) to $D_3$ at vertex $v_2$
with the partition with two elements, the first containing the edge from $v_2$ to $v_1$ and a loop at $v_2$,
and the second containing the other loop at $v_2$ (Figure~\ref{subfig:D3O}).
Let $D_3^{(I)}$ denote the digraph obtained from $D_3$ by applying Move (I) to $D_3$ at vertex $v_2$
with the partition with two elements, the first containing the edge from $v_1$ to $v_2$ and a loop at $v_2$,
and the second containing the other loop at $v_2$ (Figure~\ref{subfig:D3I}). Note that all three of these
digraphs are strongly connected.

The binary adjacency spectra of these digraphs are
\begin{align*}
    \baSpec(D_3)
        &=      \left\{ \frac{1 + \sqrt{5}}{2}, \frac{1 - \sqrt{5}}{2} \right\},
    \\
    \baSpec(D_3^{(O)}) = \baSpec(D_3^{(I)})
        &=      \{ 1 + \sqrt{2}, 1 - \sqrt{2}, 0 \},
\end{align*}
demonstrating that Moves (O) and (I) do not preserve (the nonzero elements of) $\baSpec$.

The normalized Laplace spectra of these digraphs are
\begin{align*}
    \nlSpec(D_3)
        &=      \left\{\frac{4}{3},0\right\},
    \\
    \nlSpec(D_3^{(O)})
        &=      \left\{\frac{1}{12} \left(13 + 2 \sqrt{2}\right),\frac{1}{12} \left(13-2 \sqrt{2}\right),0\right\},
    \\
    \nlSpec(D_3^{(I)})
        &=      \left\{\frac{1}{6} \left(7 + \sqrt{3}\right),\frac{1}{6} \left(7-\sqrt{3}\right),0\right\};
\end{align*}
the combinatorial Laplace spectra of these digraphs are
\begin{align*}
    \clSpec(D_3)
        &=      \left\{\frac{1}{2},0\right\},
    \\
    \clSpec(D_3^{(O)})
        &=      \left\{ \frac{9 + 2 \sqrt{3}}{28},\frac{9-2 \sqrt{3}}{28},0\right\},
    \\
    \clSpec(D_3^{(I)})
        &=      \left\{\frac{9 + 2 \sqrt{3}}{24}, \frac{9 - 2 \sqrt{3}}{24},0\right\};
\end{align*}
the binary normalized Laplace spectra of these digraphs are
\begin{align*}
    \bnlSpec(D_3)
        &=      \left\{\frac{3}{2},0\right\},
    \\
    \bnlSpec(D_3^{(O)})
        &=      \left\{\frac{1}{12} \left(13 + 2 \sqrt{2}\right),\frac{1}{12} \left(13-2 \sqrt{2}\right),0\right\},
    \\
    \bnlSpec(D_3^{(I)})
        &=      \left\{\frac{1}{6} \left(7+\sqrt{3}\right),\frac{1}{6} \left(7-\sqrt{3}\right),0\right\};
\end{align*}
and the binary combinatorial Laplace spectra of these digraphs are
\begin{align*}
    \bclSpec(D_3)
        &=      \left\{\frac{2}{3},0\right\},
    \\
    \bclSpec(D_3^{(O)})
        &=      \left\{\frac{9 + 2 \sqrt{3}}{28},\frac{ 9 - 2 \sqrt{3}}{28}, 0\right\},
    \\
    \bclSpec(D_3^{(I)})
        &=      \left\{\frac{9 + 2 \sqrt{3}}{24}, \frac{9 - 2 \sqrt{3}}{24},0\right\}.
\end{align*}

\begin{figure}
\begin{subfigure}{.4\textwidth}\centering
\begin{tikzpicture}[scale=.5]
\node (v1) at (-2,0){$v_1$};
\node (v2) at (0,0){$v_2$};
\draw[->,black] (v2) .. controls (1.8, -1) and (1.8, 1) .. (v2);
\draw[->,black] (v2) ..  controls (2.4, -2) and (2.4, 2) .. (v2);
\draw[->,black] (v2) ..  controls (-.7, .5) and (-1.3, .5) .. (v1);
\draw[->,black] (v1) ..  controls (-1.3, -.5) and (-.7, -.5) .. (v2);
\end{tikzpicture}
\caption{The digraph $D_3$.}
\label{subfig:D3}
\end{subfigure}
\hspace{.5cm}
\begin{subfigure}{.4\textwidth}\centering
\begin{tikzpicture}[scale=.5]
\node (v1) at (-2,0){$v_1$};
\node (u1) at (2,-2){$v_2^1$};
\node (u2) at (2,2){$v_2^2$};
\draw[->,black] (u2) .. controls (3.21, 4.06) and (4.37, 1.73) .. (u2);
\draw[->,black] (u1) .. controls (3.21, -4.06) and (4.37, -1.73) .. (u1);
\draw[->,black] (v1) to (u2);
\draw[->,black] (u1) .. controls (2.5, -1) and (2.5, 1) .. (u2);
\draw[->,black] (u2) .. controls (1.5, 1) and (1.5, -1) .. (u1);
\draw[->,black] (v1) .. controls (-.51, 0) and (0.68, -0.6) .. (u1);
\draw[->,black] (u1) .. controls (0.09, -1.79) and (-1.11, -1.19) .. (v1);
\end{tikzpicture}
\caption{The digraph $D_3^{(O)}$ obtained by applying Move (O) to $D_3$ at $v_2$ using the partition
whose first element contains the edge from $v_2$ to $v_1$ and a loop at $v_2$, and whose second element
contains the other loop at $v_2$.}
\label{subfig:D3O}
\end{subfigure}
\\
\begin{subfigure}{.4\textwidth}\centering
\begin{tikzpicture}[scale=.5]
\node (v1) at (-2,0){$v_1$};
\node (u1) at (2,-2){$v_2^1$};
\node (u2) at (2,2){$v_2^2$};
\draw[->,black] (u2) .. controls (3.21, 4.06) and (4.37, 1.73) .. (u2);
\draw[->,black] (u1) .. controls (3.21, -4.06) and (4.37, -1.73) .. (u1);
\draw[->,black] (u2) to (v1);
\draw[->,black] (u1) .. controls (2.5, -1) and (2.5, 1) .. (u2);
\draw[->,black] (u2) .. controls (1.5, 1) and (1.5, -1) .. (u1);
\draw[->,black] (v1) .. controls (-.51, 0) and (0.68, -0.6) .. (u1);
\draw[->,black] (u1) .. controls (0.09, -1.79) and (-1.11, -1.19) .. (v1);
\end{tikzpicture}
\caption{The digraph $D_3^{(I)}$ obtained by applying Move (I) to $D_3$ at $v_2$ using the partition
whose first element contains the edge from $v_1$ to $v_2$ and a loop at $v_2$, and whose second element
contains the other loop at $v_2$.}
\label{subfig:D3I}
\end{subfigure}
\caption{The strongly connected digraph $D_3$ and the application of Moves (O) and (I), also yielding strongly connected digraphs, from Example~\ref{ex:NegD3}.}
\label{fig:D3Counterex}
\end{figure}
\end{example}

\begin{example}
\label{ex:NegD4}
Let $D_4$ denote the  digraph used to illustrate Move (C) in \cite[p.1204]{Sorensen};
see Figure~\ref{subfig:D4}. We apply moves (R), (O), (I), and (C) as described in
Figure~\ref{fig:SCCounterex} and denote the resulting digraphs $D_4^{(R)}$, $D_4^{(O)}$, etc.
Note that these digraphs are all strongly connected.

\begin{figure}
\begin{subfigure}{.4\textwidth}\centering
\begin{tikzpicture}[scale=.5]
\node (v1) at (-1,0){$v_1$};
\node (v2) at (1,0){$v_2$};
\draw[->,black] (v1) .. controls (-3, -1.3) and (-3, 1.3) .. (v1);
\draw[->,black] (v1) ..  controls (-0.3, .5) and (0.3, .5) .. (v2);
\draw[->,black] (v2) ..  controls (0.3, -.5) and (-0.3, -.5) .. (v1);
\end{tikzpicture}
\caption{The digraph $D_4$.}
\label{subfig:D4}
\end{subfigure}
\hspace{.5cm}
\begin{subfigure}{.4\textwidth}\centering
\begin{tikzpicture}[scale=.5]
\node (v1) at (0,0){$v_1$};
\draw[->,black] (v1) .. controls (-2, -1.3) and (-2, 1.3) .. (v1);
\draw[->,black] (v1) .. controls (2, -1.3) and (2, 1.3) .. (v1);
\end{tikzpicture}
\caption{The digraph $D_4^{(R)}$ obtained by applying Move (R) with $v_r = v_2$.}
\label{subfig:D4R}
\end{subfigure}
\\
\begin{subfigure}{.4\textwidth}\centering
\begin{tikzpicture}[scale=.5]
\node (v2) at (-2,0){$v_2$};
\node (u1) at (2,2){$v_1^1$};
\node (u2) at (2,-2){$v_1^2$};
\draw[->,black] (u1) .. controls (3.21, 4.06) and (4.37, 1.73) .. (u1);
\draw[->,black] (v2) to  (u1);
\draw[->,black] (u1) to  (u2);
\draw[->,black] (v2) .. controls (-.51, 0) and (0.68, -0.6) .. (u2);
\draw[->,black] (u2) .. controls (0.09, -1.79) and (-1.11, -1.19) .. (v2);
\end{tikzpicture}
\caption{The digraph $D_4^{(O)}$.}
\label{subfig:D4O}
\end{subfigure}
\hspace{.5cm}
\begin{subfigure}{.4\textwidth}\centering
\begin{tikzpicture}[scale=.5]
\node (v2) at (-2,0){$v_2$};
\node (u1) at (2,2){$v_1^1$};
\node (u2) at (2,-2){$v_1^2$};
\draw[->,black] (u1) .. controls (3.21, 4.06) and (4.37, 1.73) .. (u1);
\draw[->,black] (u1) to  (v2);
\draw[->,black] (u2) to  (u1);
\draw[->,black] (v2) .. controls (-.51, 0) and (0.68, -0.6) .. (u2);
\draw[->,black] (u2) .. controls (0.09, -1.79) and (-1.11, -1.19) .. (v2);
\end{tikzpicture}
\caption{The digraph $D_4^{(I)}$.}
\label{subfig:D4I}
\end{subfigure}
\\
\begin{subfigure}{.4\textwidth}\centering
\begin{tikzpicture}[scale=.5]
\node (v1) at (-3,0){$v_1$};
\node (v2) at (-1,0){$v_2$};
\node (u1) at (1,0){$u_1$};
\node (u2) at (3,0){$u_2$};
\draw[->,black] (v1) .. controls (-5, -1.3) and (-5, 1.3) .. (v1);
\draw[->,black] (v1) ..  controls (-2.3, .5) and (-1.7, .5) .. (v2);
\draw[->,black] (v2) ..  controls (-1.7, -.5) and (-2.3, -.5) .. (v1);
\draw[->,black] (v2) ..  controls (-0.3, .5) and (0.3, .5) .. (u1);
\draw[->,black] (u1) ..  controls (0.3, -.5) and (-0.3, -.5) .. (v2);
\draw[->,black] (u1) ..  controls (1.7, .5) and (2.3, .5) .. (u2);
\draw[->,black] (u2) ..  controls (2.3, -.5) and (1.7, -.5) .. (u1);
\draw[->,black] (u1) ..  controls (-0.3, -2) and (2.3, -2) .. (u1);
\draw[->,black] (u2) ..  controls (5, -1.3) and (5, 1.3) .. (u2);
\end{tikzpicture}
\caption{The digraph $D_4^{(C)}$ obtained from $D_4$ by applying Move (C) at vertex $v_2$.}
\label{subfig:D4C}
\end{subfigure}
\caption{The strongly connected digraph $D_4$ and the application of Moves (R), (O), (I), and (C), yielding strongly connected digraphs, from Example~\ref{ex:NegD4}.}
\label{fig:SCCounterex}
\end{figure}

The Laplace spectra of these digraphs are
\begin{align*}
    \lSpec(D_4)
        &=          \{ 4, 0 \},
    \\
    \lSpec(D_4^{(R)})
        &=          \{ 0 \},
    \\
    \lSpec(D_4^{(O)})
        =           \lSpec(D_4^{(I)})
        &=          \{ 5, 3, 0 \},
    \\
    \lSpec(D_4^{(C)})
        &=          \{ 2(2 + \sqrt{2}), 4, 2(2 - \sqrt{2}), 0 \}.
\end{align*}
demonstrating that the nonzero elements of $\lSpec$ are not preserved by any of these moves.
The adjacency spectra are
\begin{align*}
    \aSpec(D_4)
        &=          \left\{ \frac{1 + \sqrt{5}}{2}, \frac{1 - \sqrt{5}}{2} \right\},
    \\
    \aSpec(D_4^{(R)})
        &=          \{ 2 \},
    \\
    \aSpec(D_4^{(O)})
        =           \aSpec(D_4^{(I)})
        &=          \left\{ \frac{1 + \sqrt{5}}{2}, \frac{1 - \sqrt{5}}{2}, 0 \right\},
    \\
    \aSpec(D_4^{(C)})
        &=           \{ \approx 2.35567, \approx 1.47726, \approx -1.09529, 0.26236 \},
\end{align*}
and the line adjacency spectra are
\begin{align*}
    \eSpec(D_4)
        &=          \left\{ \frac{1 + \sqrt{5}}{2}, \frac{1 - \sqrt{5}}{2}, 0 \right\},
    \\
    \eSpec(D_4^{(R)})
        &=          \{2, 0\},
    \\
    \eSpec(D_4^{(O)})
        =           \eSpec(D_4^{(I)})
        &=          \left\{ \frac{1 + \sqrt{5}}{2}, \frac{1 - \sqrt{5}}{2}, 0, 0, 0 \right\},
    \\
    \eSpec(D_4^{(C)})
        &=           \{ \approx 2.35567, \approx 1.47726, \approx -1.09529, 0.26236, 0, 0, 0, 0, 0 \}.
\end{align*}
Moves (O) and (I) preserve the nonzero elements of $\aSpec$ and $\eSpec$ by
Propositions~\ref{prop:edgedualspectrum} and \ref{prop:movesadjacencyspectrum}, while this example illustrates
that Moves (R) and (C) do not. The binary adjacency spectrum of these digraphs are given by
\begin{align*}
    \baSpec(D_4)
        &=          \left\{ \frac{1 + \sqrt{5}}{2}, \frac{1 - \sqrt{5}}{2} \right\},
    \\
    \baSpec(D_4^{(R)})
        &=          \{ 1 \},
    \\
    \baSpec(D_4^{(O)})
        =           \baSpec(D_4^{(I)})
        &=          \left\{ \frac{1 + \sqrt{5}}{2}, \frac{1 - \sqrt{5}}{2}, 0 \right\},
    \\
    \baSpec(D_4^{(C)})
        &=           \{ \approx 2.35567, \approx 1.47726, \approx -1.09529, 0.26236 \}.
\end{align*}
Hence, Moves (R) and (C) do not preserve the nonzero elements of $\baSpec$ (while the same was
demonstrated for Moves (O) and (I) in Example~\ref{ex:NegD3}; see above).

The symmetric adjacency spectra are given by
\begin{align*}
    \aSpecS(D_4)
        &=          \left\{ \frac{3 + \sqrt{5}}{2}, \frac{3 - \sqrt{5}}{2} \right\},
    \\
    \aSpecS(D_4^{(R)})
        &=          \{ 4 \},
    \\
    \aSpecS(D_4^{(O)})
        =           \aSpecS(D_4^{(I)})
        &=          \{4, 1, 0\},
    \\
    \aSpecS(D_4^{(C)})
        &=           \{ \approx 5.5492, \approx 2.1823, \approx 1.19967, \approx 0.0688326 \};
\end{align*}
the symmetric binary adjacency spectra are given by
\begin{align*}
    \baSpecS(D_4)
        &=          \left\{ \frac{3 + \sqrt{5}}{2}, \frac{3 - \sqrt{5}}{2} \right\},
    \\
    \baSpecS(D_4^{(R)})
        &=          \{ 1 \},
    \\
    \baSpecS(D_4^{(O)})
        =           \baSpecS(D_4^{(I)})
        &=          \{4, 1, 0\}
    \\
    \baSpecS(D_4^{(C)})
        &=           \{ \approx 5.5492, \approx 2.1823, \approx 1.19967, \approx 0.0688326 \};
\end{align*}
and the Hermitian adjacency spectra are given by
\begin{align*}
    \hSpec(D_4)
        &=          \left\{\frac{1}{2} \left(1 + \sqrt{5}\right),\frac{1}{2} \left(1-\sqrt{5}\right)\right\},
    \\
    \hSpec(D_4^{(R)})
        &=          \{ 1 \},
    \\
    \hSpec(D_4^{(O)}) = \hSpec(D_4^{(I)})
        &=          \left\{-\sqrt{3},\sqrt{3},1\right\},
    \\
    \hSpec(D_4^{(C)})
        &=          \{2.35567,1.47726,-1.09529,0.26236\},
\end{align*}
demonstrating that the nonzero elements of $\aSpecS$, $\baSpecS$, and $\hSpec$ are not preserved by these four moves.

The skew adjacency spectra are given by
\begin{align*}
    \sSpec(D_4)
        &=          \{0,0\},
    \\
    \sSpec(D_4^{(R)})
        &=          \{0\},
    \\
    \sSpec(D_4^{(O)}) = \sSpec(D_4^{(I)})
        &=          \left\{i \sqrt{2},-i \sqrt{2},0\right\},
    \\
    \sSpec(D_4^{(C)})
        &=          \{0,0,0,0\};
\end{align*}
the binary skew adjacency spectra are given by
\begin{align*}
    \bsSpec(D_4)
        &=          \{0,0\},
    \\
    \bsSpec(D_4^{(R)})
        &=          \{0\},
    \\
    \bsSpec(D_4^{(O)}) = \bsSpec(D_4^{(I)})
        &=          \left\{i \sqrt{2},-i \sqrt{2},0\right\},
    \\
    \bsSpec(D_4^{(C)})
        &=          \{0,0,0,0\};
\end{align*}
the skew Laplace spectra are given by
\begin{align*}
    \slSpec(D_4)
        &=          \{0,0\},
    \\
    \slSpec(D_4^{(R)})
        &=          \{0\},
    \\
    \slSpec(D_4^{(O)}) = \slSpec(D_4^{(I)})
        &=          \{i,-i,0\},
    \\
    \slSpec(D_4^{(C)})
        &=          \{0,0,0,0\};
\end{align*}
and the binary skew Laplace spectra are given by
\begin{align*}
    \bslSpec(D_4)
        &=          \{0,0\},
    \\
    \bslSpec(D_4^{(R)})
        &=          \{0\},
    \\
    \bslSpec(D_4^{(O)}) = \bslSpec(D_4^{(I)})
        &=          \{i,-i,0\},
    \\
    \bslSpec(D_4^{(C)})
        &=          \{0,0,0,0\},
\end{align*}
demonstrating that the nonzero elements of $\sSpec$, $\bsSpec$, $\slSpec$, and $\bslSpec$ are not preserved by
Moves (O) nor (I). Note that Move (C) preserves these the nonzero elements of these spectra by Proposition~\ref{prop:movesskewspectra}; the case of Move (R) will be considered in
Example~\ref{ex:NegD5} below.

The normalized Laplace spectra are given by
\begin{align*}
    \nlSpec(D_4)
        &=          \left\{\frac{3}{2},0\right\},
    \\
    \nlSpec(D_4^{(R)})
        &=          \{0\},
    \\
    \nlSpec(D_4^{(O)}) = \nlSpec(D_4^{(I)})
        &=          \left\{\frac{7}{4},\frac{3}{4},0\right\},
    \\
    \nlSpec(D_4^{(C)})
        &=          \left\{\frac{3}{2},\frac{1}{12} \left(7+\sqrt{13}\right),\frac{1}{12} \left(7-\sqrt{13}\right),0\right\};
\end{align*}
the combinatorial Laplace spectra are given by
\begin{align*}
    \clSpec(D_4)
        &=          \left\{\frac{2}{3},0\right\},
    \\
    \clSpec(D_4^{(R)})
        &=          \{0\},
    \\
    \clSpec(D_4^{(O)}) = \clSpec(D_4^{(I)})
        &=          \left\{\frac{7}{12},\frac{1}{4},0\right\},
    \\
    \clSpec(D_4^{(C)})
        &=          \left\{\frac{2 + \sqrt{2}}{9},\frac{2}{9},\frac{2-\sqrt{2}}{9},0\right\};
\end{align*}
the binary normalized Laplace spectra are given by
\begin{align*}
    \bnlSpec(D_4)
        &=          \left\{\frac{3}{2},0\right\},
    \\
    \bnlSpec(D_4^{(R)})
        &=          \{0\},
    \\
    \bnlSpec(D_4^{(O)}) = \bnlSpec(D_4^{(I)})
        &=          \left\{\frac{7}{4},\frac{3}{4},0\right\},
    \\
    \bnlSpec(D_4^{(C)})
        &=          \left\{\frac{3}{2},\frac{1}{12} \left(7+\sqrt{13}\right),\frac{1}{12} \left(7-\sqrt{13}\right),0\right\};
\end{align*}
and the binary combinatorial Laplace spectra are given by
\begin{align*}
    \bclSpec(D_4)
        &=          \left\{\frac{2}{3},0\right\},
    \\
    \bclSpec(D_4^{(R)})
        &=          \{0\},
    \\
    \bclSpec(D_4^{(O)}) = \bclSpec(D_4^{(I)})
        &=          \left\{\frac{7}{12},\frac{1}{4},0\right\},
    \\
    \bclSpec(D_4^{(C)})
        &=          \left\{\frac{2 + \sqrt{2}}{9}, \frac{2}{9}, \frac{2-\sqrt{2}}{9},0\right\},
\end{align*}
demonstrating that none of these moves preserve the nonzero elements of $\nlSpec$, $\clSpec$,
$\bnlSpec$, nor $\bclSpec$.
\end{example}

It remains only to demonstrate that Move (R) does not preserve the nonzero elements of $\sSpec$,
$\bsSpec$, $\slSpec$, nor $\bslSpec$.

\begin{example}
\label{ex:NegD5}
Let $D_5$ denote the digraph in Figure~\ref{subfig:D5} given by a single cycle of length $3$,
and let $D_5^{(R)}$ denote the result of applying Move (R) to $D_5$ at vertex $v_r$. Note that both
$D_5$ and $D_5^{(R)}$ are strongly connected.

\begin{figure}
\begin{subfigure}{.4\textwidth}\centering
\begin{tikzpicture}[scale=.5]
\node (v1) at (-2,0){$v_1$};
\node (v2) at (2,-2){$v_2$};
\node (v3) at (2,2){$v_r$};
\draw[->,black] (v1) to (v2);
\draw[->,black] (v2) to (v3);
\draw[->,black] (v3) to (v1);
\end{tikzpicture}
\caption{The digraph $D_5$.}
\label{subfig:D5}
\end{subfigure}
\hspace{.5cm}
\begin{subfigure}{.4\textwidth}\centering
\begin{tikzpicture}[scale=.5]
\node (v1) at (-1,0){$v_1$};
\node (v2) at (1,0){$v_2$};
\draw[->,black] (v1) ..  controls (-0.3, .5) and (0.3, .5) .. (v2);
\draw[->,black] (v2) ..  controls (0.3, -.5) and (-0.3, -.5) .. (v1);
\end{tikzpicture}
\caption{The digraph $D_5^{(R)}$ obtained from $D_5$ by applying Move (R) at vertex $v_r$.}
\label{subfig:D5R}
\end{subfigure}
\caption{The strongly connected digraph $D_5$ and the application of Move (R), yielding a strongly connected digraph,
from Example~\ref{ex:NegD5}.}
\label{fig:D5Counterex}
\end{figure}
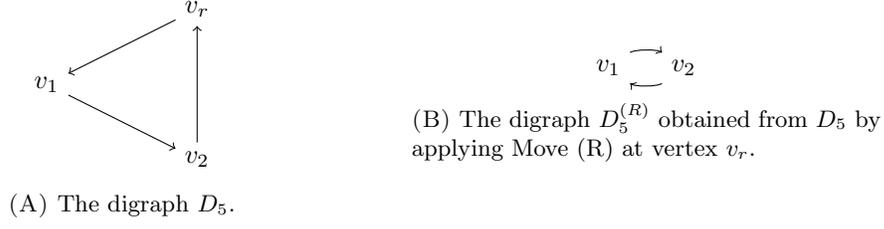

The skew adjacency spectra, binary skew adjacency spectra, skew Laplace spectra, and binary skew Laplace spectra
of these digraphs all coincide and are given by
\begin{align*}
    \sSpec(D_5) &= \bsSpec(D_5) = \slSpec(D_5) = \bslSpec(D_5)
        =  \{i \sqrt{3}, -i\sqrt{3}, 0 \}
    \\
    \sSpec(D_5^{(R)}) &= \bsSpec(D_5^{(R)}) = \slSpec(D_5^{(R)}) = \bslSpec(D_5^{(R)})
        =  \{0, 0\}.
\end{align*}
Hence, Move (R) does not preserve the nonzero elements of any of these spectra.
\end{example}

We end this section with the following example, which indicates the degree to which several features of the Laplace spectrum can change within a class of digraphs with Morita equivalent $C^\ast$-algebras. Specifically, repeated application of Move (S) yields a family of digraphs for which the number of nonzero elements of $\lSpec$, the maximum multiplicity of a nonzero element of $\lSpec$, and the maximum modulus of an element of $\lSpec$ are all unbounded.

\begin{example}
\label{ex:LapSpecUnbound}
Let $D_0$ be the digraph with $V_0 = \{w_1, w_2\}$ and $E_0 = \{e_1, e_2\}$, where $s(e_1) = w_1$, $s(e_2) = w_2$, $r(e_1) = w_2$, and $r(e_2) = w_1$; see Figure~\ref{subfig:D0}. Let $D_m$ be the digraph with vertices $V_m = V_0\cup\{ v_i : i = 1,\ldots,m \}$ and edges $E_m = E_0\cup\{ f_{1,i}, f_{2,i} : i=1,\ldots,m \}$, where for each $i$, $s(f_{1,i}) = s(f_{2,i}) = v_i$, $r(f_{1,i}) = w_1$, and $r(f_{2,i}) = w_2$; see Figure~\ref{subfig:Dm}. Then applying Move (S) to remove vertex $v_s = v_n$ from $D_m$ yields $D_{m-1}$ for $m \geq 1$.

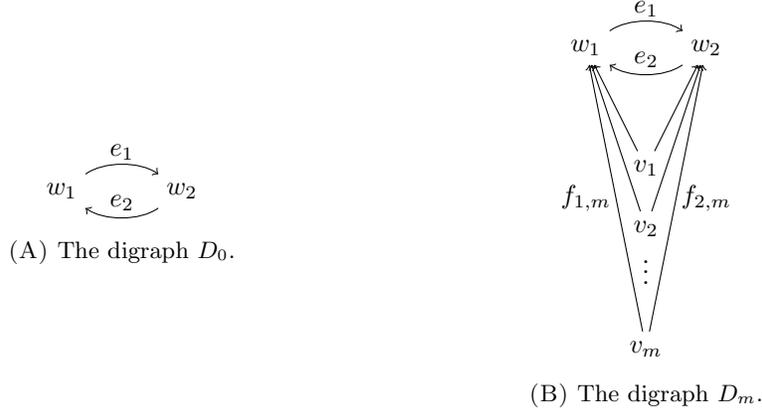
\begin{figure}
\begin{subfigure}{.4\textwidth}\centering
\begin{tikzpicture}[scale=.8]
\node (w1) at (0,0){$w_1$};
\node (w2) at (2,0){$w_2$};
\node (e1) at (1,.65){$e_1$};
\node (e2) at (1,-.2){$e_2$};
\draw[->,black] (w1) ..  controls (.7, .5) and (1.3, .5) .. (w2);
\draw[->,black] (w2) ..  controls (1.3, -.5) and (.7, -.5) .. (w1);
\end{tikzpicture}
\caption{The digraph $D_0$.}
\label{subfig:D0}
\end{subfigure}
\hspace{.5cm}
\begin{subfigure}{.4\textwidth}\centering
\begin{tikzpicture}[scale=.8]
\node (w1) at (0,0){$w_1$};
\node (w2) at (2,0){$w_2$};
\node (e1) at (1,.65){$e_1$};
\node (e2) at (1,-.2){$e_2$};
\node (f1m) at (0,-2.5){$f_{1,m}$};
\node (f2m) at (2,-2.5){$f_{2,m}$};
\node (v1) at (1,-2){$v_1$};
\node (v2) at (1,-3){$v_2$};
\node (vm) at (1,-5){$v_m$};
\node (dots) at (1,-3.6){$\vdots$};
\draw[->,black] (w1) ..  controls (.7, .5) and (1.3, .5) .. (w2);
\draw[->,black] (w2) ..  controls (1.3, -.5) and (.7, -.5) .. (w1);
\draw[->, black] (v1) to  (w1);
\draw[->, black] (v1) to  (w2);
\draw[->, black] (v2) to  (w1);
\draw[->, black] (v2) to  (w2);
\draw[->, black] (vm) to  (w1);
\draw[->, black] (vm) to  (w2);
\end{tikzpicture}
\caption{The digraph $D_m$.}
\label{subfig:Dm}
\end{subfigure}
\caption{The digraphs $D_0$ and $D_m$ ofrom Example~\ref{ex:LapSpecUnbound}.}
\label{fig:D5Counterex}
\end{figure}

Ordering the vertices and edges of $D_m$ as listed, we have
\[
    \im(D_0)    =   \begin{pmatrix} 1 & -1 \\ -1 & 1 \end{pmatrix},
    \quad
    \lm(D_0)    =   \begin{pmatrix} 2 & -2 \\ -2 & 2 \end{pmatrix},
    \quad
    \lSpec(D_0) =   \{ 4, 0 \},
\]
and
\[
    \im(D_m)  =   \begin{pmatrix} \im(D_0)    &   -I_2         &   -I_2         &   \cdots  &   -I_2         \\
                                \mathbf{0}  &   b  &   \mathbf{0}  &   \cdots  &   \mathbf{0}  \\
                                \mathbf{0}  &   \mathbf{0}  &   b  &   \cdots  &   \mathbf{0}  \\
                                \vdots      &   \vdots      &   \vdots      &   \ddots  &   \vdots      \\
                                \mathbf{0}  &   \mathbf{0}  &   \cdots      &   \mathbf{0}& b
                \end{pmatrix},
\]
where $I_2$ is the $2\times 2$ identity matrix, each $\mathbf{0}$ is a $1\times 2$ block, and $b$ is the $1\times 2$ block $(1, 1)$. Then
\[
    \lm(D_m) = \begin{pmatrix} \begin{pmatrix} m+2 & -2 \\ -2 & m+2 \end{pmatrix}
                                            &   -b^\Trn       &   -b^\Trn           &   \cdots  &   -b^\Trn           \\
                                -b   &   2    &   0  &   \cdots  &   0          \\
                                -b   &   0& 2  &   \cdots  &   0                \\
                                \vdots      &   \vdots      &   \vdots      &   \ddots  &   \vdots      \\
                                -b          &   0  &   \cdots      &   0 & 2
                \end{pmatrix},
\]
where $b$ is as above. Hence
\[
    \lSpec(D_m) = \big\{ m+4, m+2, \overbrace{2, 2, \ldots, 2}^{\text{$m - 1$ times}}, 0 \big\}.
\]
Hence, among the $D_m$, $m \geq 0$, the number of nonzero elements of $\lSpec(D_m)$ and
the maximum multiplicity and maximum modulus of an element of $\lSpec(D_m)$
are unbounded.
\end{example}


\bibliographystyle{amsplain}
\bibliography{FPS-graph-spec}

\end{document}